\documentclass[10pt,a4paper]{article}
\usepackage{amssymb}
\usepackage{amsmath}
\usepackage{graphicx}
\usepackage[T1]{fontenc} 
\usepackage[english]{babel}
\usepackage{amsthm}
\usepackage{tabulary}
\usepackage{booktabs}
\usepackage{mathtools}
\usepackage{tikz}
\usepackage{comment}
\usepackage{pst-node}
\usepackage{tikz-cd}
\usepackage{shuffle}
\usepackage{comment}
\usepackage{xcolor}
\usepackage{stmaryrd}
\usepackage{filecontents}
\usepackage{planarforest}
\newcommand{\forestA}{
\tikz[planar forest ] {

\node [b] at (0.0, 0.0) {  } 
;
}}

\newcommand{\forestB}{
\tikz[planar forest ] {

\node [b] at (0.0, 0.0) {  } 
child {node [b] at (0.0, 1.0) {  }  
}
;
}}

\newcommand{\forestC}{
\tikz[planar forest ] {

\node [b] at (0.0, 0.0) {  } 
child {node [b] at (-0.5, 1.0) {  }  
}
child {node [b] at (0.5, 1.0) {  }  
}
;
}}

\newcommand{\forestD}{
\tikz[planar forest ] {

\node [b] at (0.0, 0.0) {  } 
child {node [b] at (0.0, 1.0) {  }  
child {node [b] at (0.0, 1.0) {  }  
}
}
;
}}

\newcommand{\forestE}{
\tikz[planar forest ] {

\node [b] at (0.0, 0.0) {  } 
child {node [b] at (-1.0, 1.0) {  }  
}
child {node [b] at (0.0, 1.0) {  }  
}
child {node [b] at (1.0, 1.0) {  }  
}
;
}}

\newcommand{\forestF}{
\tikz[planar forest ] {

\node [b] at (0.0, 0.0) {  } 
child {node [b] at (-0.5, 1.0) {  }  
child {node [b] at (0.0, 1.0) {  }  
}
}
child {node [b] at (0.5, 1.0) {  }  
}
;
}}

\newcommand{\forestG}{
\tikz[planar forest ] {

\node [b] at (0.0, 0.0) {  } 
child {node [b] at (-0.5, 1.0) {  }  
}
child {node [b] at (0.5, 1.0) {  }  
child {node [b] at (0.0, 1.0) {  }  
}
}
;
}}

\newcommand{\forestH}{
\tikz[planar forest ] {

\node [b] at (0.0, 0.0) {  } 
child {node [b] at (0.0, 1.0) {  }  
child {node [b] at (-0.5, 1.0) {  }  
}
child {node [b] at (0.5, 1.0) {  }  
}
}
;
}}

\newcommand{\forestI}{
\tikz[planar forest ] {

\node [b] at (0.0, 0.0) {  } 
child {node [b] at (0.0, 1.0) {  }  
child {node [b] at (0.0, 1.0) {  }  
child {node [b] at (0.0, 1.0) {  }  
}
}
}
;
}}

\newcommand{\forestJ}{
\tikz[planar forest ] {

\draw (-1.0,1.0) edge[-] (0.0,1.0);
\draw (0.0,1.0) edge[-] (1.0,1.0);
\draw (-1.0,1.0) arc (90:270:0.5);
\draw (1.0,1.0) arc (90:-90:0.5);
\draw (-1.0,0.0) edge[-] (1.0,0.0);
\node [draw=none] at (0.0, 0.0) { } 
child {node [b] at (-1.0, 1.0) {  } edge from parent[draw=none] 
child {node [b] at (0.0, 1.0) {  }  
}
}
child {node [b] at (0.0, 1.0) {  } edge from parent[draw=none] 
}
child {node [b] at (1.0, 1.0) {  } edge from parent[draw=none] 
child {node [b] at (-0.5, 1.0) {  }  
}
child {node [b] at (0.5, 1.0) {  }  
}
}
;
}}

\newcommand{\forestK}{
\tikz[planar forest ] {

\draw (-1.0,1.0) edge[-] (0.0,1.0);
\draw (0.0,1.0) edge[-] (1.5,1.0);
\draw (-1.0,1.0) arc (90:270:0.5);
\draw (1.5,1.0) arc (90:-90:0.5);
\draw (-1.0,0.0) edge[-] (1.5,0.0);
\node [draw=none] at (0.0, 0.0) { } 
child {node [b] at (-1.0, 1.0) {  } edge from parent[draw=none] 
}
child {node [b] at (0.0, 1.0) {  } edge from parent[draw=none] 
child {node [b] at (-0.5, 1.0) {  }  
}
child {node [b] at (0.5, 1.0) {  }  
}
}
child {node [b] at (1.5, 1.0) {  } edge from parent[draw=none] 
child {node [b] at (0.0, 1.0) {  }  
}
}
;
}}

\newcommand{\forestL}{
\tikz[planar forest ] {

\draw (-1.5,1.0) edge[-] (0.0,1.0);
\draw (0.0,1.0) edge[-] (1.0,1.0);
\draw (-1.5,1.0) arc (90:270:0.5);
\draw (1.0,1.0) arc (90:-90:0.5);
\draw (-1.5,0.0) edge[-] (1.0,0.0);
\node [draw=none] at (0.0, 0.0) { } 
child {node [b] at (-1.5, 1.0) {  } edge from parent[draw=none] 
child {node [b] at (-0.5, 1.0) {  }  
}
child {node [b] at (0.5, 1.0) {  }  
}
}
child {node [b] at (0.0, 1.0) {  } edge from parent[draw=none] 
child {node [b] at (0.0, 1.0) {  }  
}
}
child {node [b] at (1.0, 1.0) {  } edge from parent[draw=none] 
}
;
}}

\newcommand{\forestM}{
\tikz[planar forest ] {

\node [b] at (0.0, 0.0) {  } 
child {node [b] at (-0.5, 1.0) {  }  
}
child {node [b] at (0.5, 1.0) {  }  
}
;
}}

\newcommand{\forestN}{
\tikz[planar forest ] {

\node [b] at (0.0, 0.0) {  } 
child {node [b] at (0.0, 1.0) {  }  
}
;
}}

\newcommand{\forestO}{
\tikz[planar forest ] {

\node [b] at (0.0, 0.0) {  } 
child {node [b] at (-0.5, 1.0) {  }  
child {node [b] at (-0.5, 1.0) {  }  
}
child {node [b] at (0.5, 1.0) {  }  
}
}
child {node [b] at (0.5, 1.0) {  }  
}
;
}}

\newcommand{\forestP}{
\tikz[planar forest ] {

\node [b] at (0.0, 0.0) {  } 
child {node [b] at (0.0, 1.0) {  }  
child {node [b] at (0.0, 1.0) {  }  
child {node [b] at (-0.5, 1.0) {  }  
}
child {node [b] at (0.5, 1.0) {  }  
}
}
}
;
}}

\newcommand{\forestQ}{
\tikz[planar forest ] {

\node [b] at (0.0, 0.0) {  } 
child {node [b] at (0.0, 1.0) {  }  
}
;
}}

\newcommand{\forestR}{
\tikz[planar forest ] {

\draw (-0.5,1.0) edge[-] (0.5,1.0);
\draw (-0.5,1.0) arc (90:270:0.5);
\draw (0.5,1.0) arc (90:-90:0.5);
\draw (-0.5,0.0) edge[-] (0.5,0.0);
\node [draw=none] at (0.0, 0.0) { } 
child {node [b] at (-0.5, 1.0) {  } edge from parent[draw=none] 
}
child {node [b] at (0.5, 1.0) {  } edge from parent[draw=none] 
}
;
}}

\newcommand{\forestS}{
\tikz[planar forest ] {

\node [b] at (0.0, 0.0) {  } 
;
}}

\newcommand{\forestT}{
\tikz[planar forest ] {

\draw (-0.5,1.0) edge[-] (0.5,1.0);
\draw (-0.5,1.0) arc (90:270:0.5);
\draw (0.5,1.0) arc (90:-90:0.5);
\draw (-0.5,0.0) edge[-] (0.5,0.0);
\node [draw=none] at (0.0, 0.0) { } 
child {node [b] at (-0.5, 1.0) {  } edge from parent[draw=none] 
child {node [b] at (0.0, 1.0) {  }  
child {node [b] at (0.0, 1.0) {  }  
}
}
}
child {node [b] at (0.5, 1.0) {  } edge from parent[draw=none] 
}
;
}}

\newcommand{\forestU}{
\tikz[planar forest ] {

\node [b] at (0.0, 0.0) {  } 
;
}}

\newcommand{\forestV}{
\tikz[planar forest ] {

\draw (-0.5,1.0) edge[-] (0.5,1.0);
\draw (-0.5,1.0) arc (90:270:0.5);
\draw (0.5,1.0) arc (90:-90:0.5);
\draw (-0.5,0.0) edge[-] (0.5,0.0);
\node [draw=none] at (0.0, 0.0) { } 
child {node [b] at (-0.5, 1.0) {  } edge from parent[draw=none] 
}
child {node [b] at (0.5, 1.0) {  } edge from parent[draw=none] 
child {node [b] at (0.0, 1.0) {  }  
child {node [b] at (0.0, 1.0) {  }  
}
}
}
;
}}

\newcommand{\forestW}{
\tikz[planar forest ] {

\node [b] at (0.0, 0.0) {  } 
;
}}

\newcommand{\forestX}{
\tikz[planar forest ] {

\draw (-0.5,1.0) edge[-] (0.5,1.0);
\draw (-0.5,1.0) arc (90:270:0.5);
\draw (0.5,1.0) arc (90:-90:0.5);
\draw (-0.5,0.0) edge[-] (0.5,0.0);
\node [draw=none] at (0.0, 0.0) { } 
child {node [b] at (-0.5, 1.0) {  } edge from parent[draw=none] 
}
child {node [b] at (0.5, 1.0) {  } edge from parent[draw=none] 
}
;
}}

\newcommand{\forestY}{
\tikz[planar forest ] {

\node [b] at (0.0, 0.0) {  } 
child {node [b] at (0.0, 1.0) {  }  
child {node [b] at (0.0, 1.0) {  }  
}
}
;
}}

\newcommand{\forestAB}{
\tikz[planar forest ] {

\draw (-0.75,1.0) edge[-] (0.75,1.0);
\draw (-0.75,1.0) arc (90:270:0.5);
\draw (0.75,1.0) arc (90:-90:0.5);
\draw (-0.75,0.0) edge[-] (0.75,0.0);
\node [draw=none] at (0.0, 0.0) { } 
child {node [b] at (-0.75, 1.0) {  } edge from parent[draw=none] 
child {node [b] at (0.0, 1.0) {  }  
}
}
child {node [b] at (0.75, 1.0) {  } edge from parent[draw=none] 
child {node [b] at (-0.5, 1.0) {  }  
}
child {node [b] at (0.5, 1.0) {  }  
}
}
;
}}

\newcommand{\forestBB}{
\tikz[planar forest ] {

\node [b] at (0.0, 0.0) {  } 
;
}}

\newcommand{\forestCB}{
\tikz[planar forest ] {

\draw (0.0,1.0) arc (90:270:0.5);
\draw (0.0,1.0) arc (90:-90:0.5);
\node [draw=none] at (0.0, 0.0) { } 
child {node [b] at (0.0, 1.0) {  } edge from parent[draw=none] 
}
;
}}

\newcommand{\forestDB}{
\tikz[planar forest ] {

\node [b] at (0.0, 0.0) {  } 
child {node [b] at (0.0, 1.0) {  }  
}
;
}}

\newcommand{\forestEB}{
\tikz[planar forest ] {

\draw (-0.75,1.0) edge[-] (0.75,1.0);
\draw (-0.75,1.0) arc (90:270:0.5);
\draw (0.75,1.0) arc (90:-90:0.5);
\draw (-0.75,0.0) edge[-] (0.75,0.0);
\node [draw=none] at (0.0, 0.0) { } 
child {node [b] at (-0.75, 1.0) {  } edge from parent[draw=none] 
child {node [b] at (0.0, 1.0) {  }  
}
}
child {node [b] at (0.75, 1.0) {  } edge from parent[draw=none] 
child {node [b] at (-0.5, 1.0) {  }  
}
child {node [b] at (0.5, 1.0) {  }  
}
}
;
}}

\newcommand{\forestFB}{
\tikz[planar forest ] {

\draw (0.0,1.0) arc (90:270:0.5);
\draw (0.0,1.0) arc (90:-90:0.5);
\node [draw=none] at (0.0, 0.0) { } 
child {node [b] at (0.0, 1.0) {  } edge from parent[draw=none] 
child {node [b] at (0.0, 1.0) {  }  
}
}
;
}}

\newcommand{\forestGB}{
\tikz[planar forest ] {

\node [b] at (0.0, 0.0) {  } 
child {node [b] at (0.0, 1.0) {  }  
}
;
}}

\newcommand{\forestHB}{
\tikz[planar forest ] {

\draw (-0.75,1.0) edge[-] (0.75,1.0);
\draw (-0.75,1.0) arc (90:270:0.5);
\draw (0.75,1.0) arc (90:-90:0.5);
\draw (-0.75,0.0) edge[-] (0.75,0.0);
\node [draw=none] at (0.0, 0.0) { } 
child {node [b] at (-0.75, 1.0) {  } edge from parent[draw=none] 
child {node [b] at (0.0, 1.0) {  }  
}
}
child {node [b] at (0.75, 1.0) {  } edge from parent[draw=none] 
child {node [b] at (-0.5, 1.0) {  }  
}
child {node [b] at (0.5, 1.0) {  }  
}
}
;
}}

\newcommand{\forestIB}{
\tikz[planar forest ] {

\draw (0.0,1.0) arc (90:270:0.5);
\draw (0.0,1.0) arc (90:-90:0.5);
\node [draw=none] at (0.0, 0.0) { } 
child {node [b] at (0.0, 1.0) {  } edge from parent[draw=none] 
}
;
}}

\newcommand{\forestJB}{
\tikz[planar forest ] {

\node [b] at (0.0, 0.0) {  } 
child {node [b] at (-0.5, 1.0) {  }  
}
child {node [b] at (0.5, 1.0) {  }  
}
;
}}

\newcommand{\forestKB}{
\tikz[planar forest ] {

\draw (-0.75,1.0) edge[-] (0.75,1.0);
\draw (-0.75,1.0) arc (90:270:0.5);
\draw (0.75,1.0) arc (90:-90:0.5);
\draw (-0.75,0.0) edge[-] (0.75,0.0);
\node [draw=none] at (0.0, 0.0) { } 
child {node [b] at (-0.75, 1.0) {  } edge from parent[draw=none] 
child {node [b] at (0.0, 1.0) {  }  
}
}
child {node [b] at (0.75, 1.0) {  } edge from parent[draw=none] 
child {node [b] at (-0.5, 1.0) {  }  
}
child {node [b] at (0.5, 1.0) {  }  
}
}
;
}}

\newcommand{\forestLB}{
\tikz[planar forest ] {

\draw (0.0,1.0) arc (90:270:0.5);
\draw (0.0,1.0) arc (90:-90:0.5);
\node [draw=none] at (0.0, 0.0) { } 
child {node [b] at (0.0, 1.0) {  } edge from parent[draw=none] 
}
;
}}

\newcommand{\forestMB}{
\tikz[planar forest ] {

\node [b] at (0.0, 0.0) {  } 
child {node [b] at (0.0, 1.0) {  }  
child {node [b] at (0.0, 1.0) {  }  
}
}
;
}}

\newcommand{\forestNB}{
\tikz[planar forest ] {

\draw (-0.5,1.0) edge[-] (0.5,1.0);
\draw (-0.5,1.0) arc (90:270:0.5);
\draw (0.5,1.0) arc (90:-90:0.5);
\draw (-0.5,0.0) edge[-] (0.5,0.0);
\node [draw=none] at (0.0, 0.0) { } 
child {node [b] at (-0.5, 1.0) {  } edge from parent[draw=none] 
child {node [b] at (0.0, 1.0) {  }  
}
}
child {node [b] at (0.5, 1.0) {  } edge from parent[draw=none] 
}
;
}}

\newcommand{\forestOB}{
\tikz[planar forest ] {

\node [b] at (0.0, 0.0) {  } 
child {node [b] at (-0.5, 1.0) {  }  
}
child {node [b] at (0.5, 1.0) {  }  
}
;
}}

\newcommand{\forestPB}{
\tikz[planar forest ] {

\draw (-0.75,1.0) edge[-] (0.75,1.0);
\draw (-0.75,1.0) arc (90:270:0.5);
\draw (0.75,1.0) arc (90:-90:0.5);
\draw (-0.75,0.0) edge[-] (0.75,0.0);
\node [draw=none] at (0.0, 0.0) { } 
child {node [b] at (-0.75, 1.0) {  } edge from parent[draw=none] 
child {node [b,label={ [label distance=-1mm]0:{ \scriptsize v } }] at (-0.5, 1.0) {  }  
}
child {node [b] at (0.5, 1.0) {  }  
}
}
child {node [b] at (0.75, 1.0) {  } edge from parent[draw=none] 
child {node [b] at (0.0, 1.0) {  }  
}
}
;
}}

\newcommand{\forestQB}{
\tikz[planar forest ] {

\node [b] at (0.0, 0.0) {  } 
child {node [b] at (-0.5, 1.0) {  }  
child {node [b] at (0.0, 1.0) {  }  
}
}
child {node [b] at (0.5, 1.0) {  }  
}
;
}}

\newcommand{\forestRB}{
\tikz[planar forest ] {

\draw (-0.5,1.0) edge[-] (0.5,1.0);
\draw (-0.5,1.0) arc (90:270:0.5);
\draw (0.5,1.0) arc (90:-90:0.5);
\draw (-0.5,0.0) edge[-] (0.5,0.0);
\node [draw=none] at (0.0, 0.0) { } 
child {node [b] at (-0.5, 1.0) {  } edge from parent[draw=none] 
child {node [b] at (0.0, 1.0) {  }  
}
}
child {node [b] at (0.5, 1.0) {  } edge from parent[draw=none] 
}
;
}}

\newcommand{\forestSB}{
\tikz[planar forest ] {

\draw (-0.75,1.0) edge[-] (0.75,1.0);
\draw (-0.75,1.0) arc (90:270:0.5);
\draw (0.75,1.0) arc (90:-90:0.5);
\draw (-0.75,0.0) edge[-] (0.75,0.0);
\node [draw=none] at (0.0, 0.0) { } 
child {node [b] at (-0.75, 1.0) {  } edge from parent[draw=none] 
child {node [b] at (-0.5, 1.0) {  }  
child {node [b] at (0.0, 1.0) {  }  
child {node [b] at (-0.5, 1.0) {  }  
}
child {node [b] at (0.5, 1.0) {  }  
}
}
}
child {node [b] at (0.5, 1.0) {  }  
}
}
child {node [b] at (0.75, 1.0) {  } edge from parent[draw=none] 
child {node [b] at (0.0, 1.0) {  }  
}
}
;
}}

\newcommand{\forestTB}{
\tikz[planar forest ] {

\node [b] at (0.0, 0.0) {  } 
child {node [b] at (-0.5, 1.0) {  }  
child {node [b] at (0.0, 1.0) {  }  
}
}
child {node [b] at (0.5, 1.0) {  }  
}
;
}}

\newcommand{\forestUB}{
\tikz[planar forest ] {

\node [b,label={ [label distance=-1mm]0:{ \scriptsize a } }] at (0.0, 0.0) {  } 
child {node [b,label={ [label distance=-1mm]0:{ \scriptsize b } }] at (-0.5, 1.0) {  }  
child {node [b,label={ [label distance=-1mm]0:{ \scriptsize d } }] at (0.0, 1.0) {  }  
}
}
child {node [b,label={ [label distance=-1mm]0:{ \scriptsize c } }] at (0.5, 1.0) {  }  
}
;
}}

\newcommand{\forestVB}{
\tikz[planar forest ] {

\draw (-1.0,1.0) edge[-] (0.0,1.0);
\draw (0.0,1.0) edge[-] (1.0,1.0);
\draw (-1.0,1.0) arc (90:270:0.5);
\draw (1.0,1.0) arc (90:-90:0.5);
\draw (-1.0,0.0) edge[-] (1.0,0.0);
\node [draw=none] at (0.0, 0.0) { } 
child {node [b,label={ [label distance=-1mm]0:{ \scriptsize b } }] at (-1.0, 1.0) {  } edge from parent[draw=none] 
}
child {node [b,label={ [label distance=-1mm]0:{ \scriptsize a } }] at (0.0, 1.0) {  } edge from parent[draw=none] 
child {node [b,label={ [label distance=-1mm]0:{ \scriptsize c } }] at (0.0, 1.0) {  }  
}
}
child {node [b,label={ [label distance=-1mm]0:{ \scriptsize d } }] at (1.0, 1.0) {  } edge from parent[draw=none] 
}
;
}}

\newcommand{\forestWB}{
\tikz[planar forest ] {

\node [b,label={ [label distance=-1mm]0:{ \scriptsize d } }] at (0.0, 0.0) {  } 
;
}}

\newcommand{\forestXB}{
\tikz[planar forest ] {

\node [b,label={ [label distance=-1mm]0:{ \scriptsize d } }] at (0.0, 0.0) {  } 
;
}}

\newcommand{\forestYB}{
\tikz[planar forest ] {

\draw (-1.0,1.0) edge[-] (0.0,1.0);
\draw (0.0,1.0) edge[-] (1.0,1.0);
\draw (-1.0,1.0) arc (90:270:0.5);
\draw (1.0,1.0) arc (90:-90:0.5);
\draw (-1.0,0.0) edge[-] (1.0,0.0);
\node [draw=none] at (0.0, 0.0) { } 
child {node [b,label={ [label distance=-1mm]0:{ \scriptsize a } }] at (-1.0, 1.0) {  } edge from parent[draw=none] 
child {node [b,label={ [label distance=-1mm]0:{ \scriptsize d } }] at (0.0, 1.0) {  }  
}
}
child {node [b,label={ [label distance=-1mm]0:{ \scriptsize b } }] at (0.0, 1.0) {  } edge from parent[draw=none] 
child {node [b,label={ [label distance=-1mm]0:{ \scriptsize e } }] at (0.0, 1.0) {  }  
}
}
child {node [b,label={ [label distance=-1mm]0:{ \scriptsize c } }] at (1.0, 1.0) {  } edge from parent[draw=none] 
}
;
}}

\newcommand{\forestAC}{
\tikz[planar forest ] {

\draw (0.0,1.0) arc (90:270:0.5);
\draw (0.0,1.0) arc (90:-90:0.5);
\node [draw=none] at (0.0, 0.0) { } 
child {node [b,label={ [label distance=-1mm]0:{ \scriptsize f } }] at (0.0, 1.0) {  } edge from parent[draw=none] 
}
;
}}

\newcommand{\forestBC}{
\tikz[planar forest ] {

\node [b,label={ [label distance=-1mm]0:{ \scriptsize g } }] at (0.0, 0.0) {  } 
child {node [b,label={ [label distance=-1mm]0:{ \scriptsize h } }] at (0.0, 1.0) {  }  
}
;
}}

\newcommand{\forestCC}{
\tikz[planar forest ] {

\node [b,label={ [label distance=-1mm]0:{ \scriptsize h } }] at (0.0, 0.0) {  } 
;
}}

\newcommand{\forestDC}{
\tikz[planar forest ] {

\node [b,label={ [label distance=-1mm]0:{ \scriptsize g } }] at (0.0, 0.0) {  } 
;
}}

\newcommand{\forestEC}{
\tikz[planar forest ] {

\node [b,label={ [label distance=-1mm]0:{ \scriptsize f } }] at (0.0, 0.0) {  } 
;
}}

\newcommand{\forestFC}{
\tikz[planar forest ] {

\draw (-1.0,1.0) edge[-] (0.0,1.0);
\draw (0.0,1.0) edge[-] (1.0,1.0);
\draw (-1.0,1.0) arc (90:270:0.5);
\draw (1.0,1.0) arc (90:-90:0.5);
\draw (-1.0,0.0) edge[-] (1.0,0.0);
\node [draw=none] at (0.0, 0.0) { } 
child {node [b,label={ [label distance=-1mm]0:{ \scriptsize a } }] at (-1.0, 1.0) {  } edge from parent[draw=none] 
child {node [b,label={ [label distance=-1mm]0:{ \scriptsize d } }] at (0.0, 1.0) {  }  
}
}
child {node [b,label={ [label distance=-1mm]0:{ \scriptsize b } }] at (0.0, 1.0) {  } edge from parent[draw=none] 
child {node [b,label={ [label distance=-1mm]0:{ \scriptsize e } }] at (0.0, 1.0) {  }  
}
}
child {node [b,label={ [label distance=-1mm]0:{ \scriptsize c } }] at (1.0, 1.0) {  } edge from parent[draw=none] 
}
;
}}

\newcommand{\forestGC}{
\tikz[planar forest ] {

\node [b,label={ [label distance=-1mm]0:{ \scriptsize f } }] at (0.0, 0.0) {  } 
;
}}

\newcommand{\forestHC}{
\tikz[planar forest ] {

\draw (-1.0,1.0) edge[-] (0.0,1.0);
\draw (0.0,1.0) edge[-] (1.0,1.0);
\draw (-1.0,1.0) arc (90:270:0.5);
\draw (1.0,1.0) arc (90:-90:0.5);
\draw (-1.0,0.0) edge[-] (1.0,0.0);
\node [draw=none] at (0.0, 0.0) { } 
child {node [b,label={ [label distance=-1mm]0:{ \scriptsize a } }] at (-1.0, 1.0) {  } edge from parent[draw=none] 
child {node [b,label={ [label distance=-1mm]0:{ \scriptsize d } }] at (0.0, 1.0) {  }  
}
}
child {node [b,label={ [label distance=-1mm]0:{ \scriptsize b } }] at (0.0, 1.0) {  } edge from parent[draw=none] 
child {node [b,label={ [label distance=-1mm]0:{ \scriptsize e } }] at (0.0, 1.0) {  }  
}
}
child {node [b,label={ [label distance=-1mm]0:{ \scriptsize c } }] at (1.0, 1.0) {  } edge from parent[draw=none] 
}
;
}}

\newcommand{\forestIC}{
\tikz[planar forest ] {

\node [b,label={ [label distance=-1mm]0:{ \scriptsize f } }] at (0.0, 0.0) {  } 
;
}}

\newcommand{\forestJC}{
\tikz[planar forest ] {

\node [b,label={ [label distance=-1mm]0:{ \scriptsize c } }] at (0.0, 0.0) {  } 
child {node [b,label={ [label distance=-1mm]0:{ \scriptsize b } }] at (0.0, 1.0) {  }  
child {node [b,label={ [label distance=-1mm]0:{ \scriptsize a } }] at (-0.5, 1.0) {  }  
child {node [b,label={ [label distance=-1mm]0:{ \scriptsize d } }] at (0.0, 1.0) {  }  
}
}
child {node [b,label={ [label distance=-1mm]0:{ \scriptsize e } }] at (0.5, 1.0) {  }  
}
}
;
}}

\newcommand{\forestKC}{
\tikz[planar forest ] {

\node [b,label={ [label distance=-1mm]0:{ \scriptsize c } }] at (0.0, 0.0) {  } 
child {node [b,label={ [label distance=-1mm]0:{ \scriptsize b } }] at (0.0, 1.0) {  }  
child {node [b,label={ [label distance=-1mm]0:{ \scriptsize a } }] at (-0.5, 1.0) {  }  
child {node [b,label={ [label distance=-1mm]0:{ \scriptsize d } }] at (0.0, 1.0) {  }  
}
}
child {node [b,label={ [label distance=-1mm]0:{ \scriptsize e } }] at (0.5, 1.0) {  }  
}
}
;
}}

\newcommand{\forestLC}{
\tikz[planar forest ] {

\node [b,label={ [label distance=-1mm]0:{ \scriptsize c } }] at (0.0, 0.0) {  } 
child {node [b,label={ [label distance=-1mm]0:{ \scriptsize b } }] at (0.0, 1.0) {  }  
child {node [b,label={ [label distance=-1mm]0:{ \scriptsize a } }] at (-0.5, 1.0) {  }  
child {node [b,label={ [label distance=-1mm]0:{ \scriptsize d } }] at (0.0, 1.0) {  }  
}
}
child {node [b,label={ [label distance=-1mm]0:{ \scriptsize e } }] at (0.5, 1.0) {  }  
}
}
;
}}

\newcommand{\forestMC}{
\tikz[planar forest ] {

\node [b,label={ [label distance=-1mm]0:{ \scriptsize a } }] at (0.0, 0.0) {  } 
child {node [b,label={ [label distance=-1mm]0:{ \scriptsize d } }] at (0.0, 1.0) {  }  
}
;
}}

\newcommand{\forestNC}{
\tikz[planar forest ] {

\node [b,label={ [label distance=-1mm]0:{ \scriptsize c } }] at (0.0, 0.0) {  } 
child {node [b,label={ [label distance=-1mm]0:{ \scriptsize b } }] at (0.0, 1.0) {  }  
child {node [b,label={ [label distance=-1mm]0:{ \scriptsize e } }] at (0.0, 1.0) {  }  
}
}
;
}}

\newcommand{\forestOC}{
\tikz[planar forest ] {

\node [b,label={ [label distance=-1mm]0:{ \scriptsize a } }] at (0.0, 0.0) {  } 
child {node [b,label={ [label distance=-1mm]0:{ \scriptsize d } }] at (0.0, 1.0) {  }  
}
;
}}

\newcommand{\forestPC}{
\tikz[planar forest ] {

\node [b,label={ [label distance=-1mm]0:{ \scriptsize b } }] at (0.0, 0.0) {  } 
child {node [b,label={ [label distance=-1mm]0:{ \scriptsize e } }] at (0.0, 1.0) {  }  
}
;
}}

\newcommand{\forestQC}{
\tikz[planar forest ] {

\node [b,label={ [label distance=-1mm]0:{ \scriptsize c } }] at (0.0, 0.0) {  } 
;
}}

\newcommand{\forestRC}{
\tikz[planar forest ] {

\node [b,label={ [label distance=-1mm]0:{ \scriptsize a } }] at (0.0, 0.0) {  } 
child {node [b,label={ [label distance=-1mm]0:{ \scriptsize d } }] at (0.0, 1.0) {  }  
}
;
}}

\newcommand{\forestSC}{
\tikz[planar forest ] {

\node [b,label={ [label distance=-1mm]0:{ \scriptsize a } }] at (0.0, 0.0) {  } 
child {node [b,label={ [label distance=-1mm]0:{ \scriptsize e } }] at (0.0, 1.0) {  }  
}
;
}}

\newcommand{\forestTC}{
\tikz[planar forest ] {

\node [b,label={ [label distance=-1mm]0:{ \scriptsize b } }] at (0.0, 0.0) {  } 
child {node [b,label={ [label distance=-1mm]0:{ \scriptsize e } }] at (0.0, 1.0) {  }  
}
;
}}

\newcommand{\forestUC}{
\tikz[planar forest ] {

\node [b,label={ [label distance=-1mm]0:{ \scriptsize b } }] at (0.0, 0.0) {  } 
child {node [b,label={ [label distance=-1mm]0:{ \scriptsize e } }] at (0.0, 1.0) {  }  
}
;
}}

\newcommand{\forestVC}{
\tikz[planar forest ] {

\node [b,label={ [label distance=-1mm]0:{ \scriptsize c } }] at (0.0, 0.0) {  } 
;
}}

\newcommand{\forestWC}{
\tikz[planar forest ] {

\node [b,label={ [label distance=-1mm]0:{ \scriptsize a } }] at (0.0, 0.0) {  } 
child {node [b,label={ [label distance=-1mm]0:{ \scriptsize d } }] at (0.0, 1.0) {  }  
}
;
}}

\newcommand{\forestXC}{
\tikz[planar forest ] {

\node [b,label={ [label distance=-1mm]0:{ \scriptsize d } }] at (0.0, 0.0) {  } 
;
}}

\newcommand{\forestYC}{
\tikz[planar forest ] {

\node [b,label={ [label distance=-1mm]0:{ \scriptsize a } }] at (0.0, 0.0) {  } 
;
}}

\newcommand{\forestAD}{
\tikz[planar forest ] {

\node [b,label={ [label distance=-1mm]0:{ \scriptsize b } }] at (0.0, 0.0) {  } 
child {node [b,label={ [label distance=-1mm]0:{ \scriptsize e } }] at (0.0, 1.0) {  }  
}
;
}}

\newcommand{\forestBD}{
\tikz[planar forest ] {

\node [b,label={ [label distance=-1mm]0:{ \scriptsize e } }] at (0.0, 0.0) {  } 
;
}}

\newcommand{\forestCD}{
\tikz[planar forest ] {

\node [b,label={ [label distance=-1mm]0:{ \scriptsize b } }] at (0.0, 0.0) {  } 
;
}}

\newcommand{\forestDD}{
\tikz[planar forest ] {

\node [b,label={ [label distance=-1mm]0:{ \scriptsize c } }] at (0.0, 0.0) {  } 
;
}}

\newcommand{\forestED}{
\tikz[planar forest ] {

\node [b,label={ [label distance=-1mm]0:{ \scriptsize d } }] at (0.0, 0.0) {  } 
;
}}

\newcommand{\forestFD}{
\tikz[planar forest ] {

\node [b,label={ [label distance=-1mm]0:{ \scriptsize a } }] at (0.0, 0.0) {  } 
;
}}

\newcommand{\forestGD}{
\tikz[planar forest ] {

\node [b,label={ [label distance=-1mm]0:{ \scriptsize d } }] at (0.0, 0.0) {  } 
;
}}

\newcommand{\forestHD}{
\tikz[planar forest ] {

\node [b,label={ [label distance=-1mm]0:{ \scriptsize a } }] at (0.0, 0.0) {  } 
;
}}

\newcommand{\forestID}{
\tikz[planar forest ] {

\node [b,label={ [label distance=-1mm]0:{ \scriptsize e } }] at (0.0, 0.0) {  } 
;
}}

\newcommand{\forestJD}{
\tikz[planar forest ] {

\node [b,label={ [label distance=-1mm]0:{ \scriptsize b } }] at (0.0, 0.0) {  } 
;
}}

\newcommand{\forestKD}{
\tikz[planar forest ] {

\node [b,label={ [label distance=-1mm]0:{ \scriptsize e } }] at (0.0, 0.0) {  } 
;
}}

\newcommand{\forestLD}{
\tikz[planar forest ] {

\node [b,label={ [label distance=-1mm]0:{ \scriptsize b } }] at (0.0, 0.0) {  } 
;
}}

\newcommand{\forestMD}{
\tikz[planar forest ] {

\node [b,label={ [label distance=-1mm]0:{ \scriptsize c } }] at (0.0, 0.0) {  } 
;
}}

\newcommand{\forestND}{
\tikz[planar forest ] {

\node [b,label={ [label distance=-1mm]0:{ \scriptsize d } }] at (0.0, 0.0) {  } 
;
}}

\newcommand{\forestOD}{
\tikz[planar forest ] {

\node [b,label={ [label distance=-1mm]0:{ \scriptsize a } }] at (0.0, 0.0) {  } 
;
}}

\newcommand{\forestPD}{
\tikz[planar forest ] {

\node [b,label={ [label distance=-1mm]0:{ \scriptsize d } }] at (0.0, 0.0) {  } 
;
}}

\newcommand{\forestQD}{
\tikz[planar forest ] {

\node [b,label={ [label distance=-1mm]0:{ \scriptsize a } }] at (0.0, 0.0) {  } 
;
}}

\newcommand{\forestRD}{
\tikz[planar forest ] {

\node [b,label={ [label distance=-1mm]0:{ \scriptsize e } }] at (0.0, 0.0) {  } 
;
}}

\newcommand{\forestSD}{
\tikz[planar forest ] {

\node [b,label={ [label distance=-1mm]0:{ \scriptsize b } }] at (0.0, 0.0) {  } 
;
}}

\newcommand{\forestTD}{
\tikz[planar forest ] {

\node [b,label={ [label distance=-1mm]0:{ \scriptsize e } }] at (0.0, 0.0) {  } 
;
}}

\newcommand{\forestUD}{
\tikz[planar forest ] {

\node [b,label={ [label distance=-1mm]0:{ \scriptsize b } }] at (0.0, 0.0) {  } 
;
}}

\newcommand{\forestVD}{
\tikz[planar forest ] {

\node [b,label={ [label distance=-1mm]0:{ \scriptsize c } }] at (0.0, 0.0) {  } 
;
}}

\newcommand{\forestWD}{
\tikz[planar forest ] {

\node [b,label={ [label distance=-1mm]0:{ \scriptsize f } }] at (0.0, 0.0) {  } 
;
}}

\newcommand{\forestXD}{
\tikz[planar forest ] {

\node [b,label={ [label distance=-1mm]0:{ \scriptsize h } }] at (0.0, 0.0) {  } 
;
}}

\newcommand{\forestYD}{
\tikz[planar forest ] {

\node [b,label={ [label distance=-1mm]0:{ \scriptsize g } }] at (0.0, 0.0) {  } 
;
}}

\newcommand{\forestAE}{
\tikz[planar forest ] {

\node [b] at (0.0, 0.0) {  } 
child {node [b] at (-0.5, 1.0) {  }  
child {node [b] at (0.0, 1.0) {  }  
}
}
child {node [b] at (0.5, 1.0) {  }  
}
;
}}

\newcommand{\forestBE}{
\tikz[planar forest ] {

\node [b] at (0.0, 0.0) {  } 
child {node [b] at (-0.5, 1.0) {  }  
}
child {node [b] at (0.5, 1.0) {  }  
child {node [b] at (0.0, 1.0) {  }  
}
}
;
}}

\newcommand{\forestCE}{
\tikz[planar forest ] {

\node [b] at (0.0, 0.0) {  } 
child {node [b] at (0.0, 1.0) {  }  
}
;
}}

\newcommand{\forestDE}{
\tikz[planar forest ] {

\node [b] at (0.0, 0.0) {  } 
child {node [b] at (-1.0, 1.0) {  }  
}
child {node [b] at (0.0, 1.0) {  }  
child {node [b] at (-0.5, 1.0) {  }  
}
child {node [b] at (0.5, 1.0) {  }  
}
}
child {node [b] at (1.0, 1.0) {  }  
}
;
}}

\newcommand{\forestEE}{
\tikz[planar forest ] {

\node [b] at (0.0, 0.0) {  } 
child {node [b] at (-1.5, 1.0) {  }  
child {node [b] at (0.0, 1.0) {  }  
}
}
child {node [b] at (-0.5, 1.0) {  }  
}
child {node [b] at (0.5, 1.0) {  }  
child {node [b] at (-0.5, 1.0) {  }  
}
child {node [b] at (0.5, 1.0) {  }  
}
}
child {node [b] at (1.5, 1.0) {  }  
}
;
}}

\newcommand{\forestFE}{
\tikz[planar forest ] {

\node [b] at (0.0, 0.0) {  } 
child {node [b] at (-1.5, 1.0) {  }  
child {node [b] at (0.0, 1.0) {  }  
child {node [b] at (0.0, 1.0) {  }  
}
}
}
child {node [b] at (0.0, 1.0) {  }  
child {node [b] at (-0.5, 1.0) {  }  
}
child {node [b] at (0.5, 1.0) {  }  
}
}
child {node [b] at (1.0, 1.0) {  }  
}
;
}}

\newcommand{\forestGE}{
\tikz[planar forest ] {

\node [b] at (0.0, 0.0) {  } 
child {node [b] at (-1.0, 1.0) {  }  
}
child {node [b] at (0.0, 1.0) {  }  
child {node [b] at (-1.0, 1.0) {  }  
child {node [b] at (0.0, 1.0) {  }  
}
}
child {node [b] at (0.0, 1.0) {  }  
}
child {node [b] at (1.0, 1.0) {  }  
}
}
child {node [b] at (1.0, 1.0) {  }  
}
;
}}

\newcommand{\forestHE}{
\tikz[planar forest ] {

\node [b] at (0.0, 0.0) {  } 
child {node [b] at (-1.0, 1.0) {  }  
}
child {node [b] at (0.0, 1.0) {  }  
child {node [b] at (-0.5, 1.0) {  }  
child {node [b] at (0.0, 1.0) {  }  
child {node [b] at (0.0, 1.0) {  }  
}
}
}
child {node [b] at (0.5, 1.0) {  }  
}
}
child {node [b] at (1.0, 1.0) {  }  
}
;
}}

\newcommand{\forestIE}{
\tikz[planar forest ] {

\node [b] at (0.0, 0.0) {  } 
child {node [b] at (-1.0, 1.0) {  }  
}
child {node [b] at (0.0, 1.0) {  }  
child {node [b] at (-0.5, 1.0) {  }  
}
child {node [b] at (0.5, 1.0) {  }  
child {node [b] at (0.0, 1.0) {  }  
child {node [b] at (0.0, 1.0) {  }  
}
}
}
}
child {node [b] at (1.0, 1.0) {  }  
}
;
}}

\newcommand{\forestJE}{
\tikz[planar forest ] {

\node [b] at (0.0, 0.0) {  } 
child {node [b] at (-1.0, 1.0) {  }  
}
child {node [b] at (0.0, 1.0) {  }  
child {node [b] at (-0.5, 1.0) {  }  
}
child {node [b] at (0.5, 1.0) {  }  
}
}
child {node [b] at (1.5, 1.0) {  }  
child {node [b] at (0.0, 1.0) {  }  
child {node [b] at (0.0, 1.0) {  }  
}
}
}
;
}}

\newcommand{\forestKE}{
\tikz[planar forest ] {

\node [b] at (0.0, 0.0) {  } 
;
}}

\newcommand{\forestLE}{
\tikz[planar forest ] {

\node [b] at (0.0, 0.0) {  } 
child {node [b] at (0.0, 1.0) {  }  
}
;
}}

\newcommand{\forestME}{
\tikz[planar forest ] {

\node [b] at (0.0, 0.0) {  } 
child {node [b] at (-0.5, 1.0) {  }  
}
child {node [b] at (0.5, 1.0) {  }  
}
;
}}

\newcommand{\forestNE}{
\tikz[planar forest ] {

\node [b] at (0.0, 0.0) {  } 
child {node [b] at (-1.5, 1.0) {  }  
}
child {node [b] at (-0.5, 1.0) {  }  
child {node [b] at (0.0, 1.0) {  }  
}
}
child {node [b] at (0.5, 1.0) {  }  
}
child {node [b] at (1.5, 1.0) {  }  
}
;
}}

\newcommand{\forestOE}{
\tikz[planar forest ] {

\node [b] at (0.0, 0.0) {  } 
child {node [b] at (-1.5, 1.0) {  }  
child {node [b] at (0.0, 1.0) {  }  
}
}
child {node [b] at (-0.5, 1.0) {  }  
}
child {node [b] at (0.5, 1.0) {  }  
}
child {node [b] at (1.5, 1.0) {  }  
}
;
}}

\newcommand{\forestPE}{
\tikz[planar forest ] {

\node [b] at (0.0, 0.0) {  } 
child {node [b] at (-0.5, 1.0) {  }  
child {node [b] at (-0.5, 1.0) {  }  
}
child {node [b] at (0.5, 1.0) {  }  
child {node [b] at (0.0, 1.0) {  }  
}
}
}
child {node [b] at (0.5, 1.0) {  }  
}
;
}}

\newcommand{\forestQE}{
\tikz[planar forest ] {

\node [b] at (0.0, 0.0) {  } 
child {node [b] at (-0.5, 1.0) {  }  
child {node [b] at (-0.5, 1.0) {  }  
child {node [b] at (0.0, 1.0) {  }  
}
}
child {node [b] at (0.5, 1.0) {  }  
}
}
child {node [b] at (0.5, 1.0) {  }  
}
;
}}

\newcommand{\forestRE}{
\tikz[planar forest ] {

\node [b] at (0.0, 0.0) {  } 
child {node [b] at (-0.5, 1.0) {  }  
}
child {node [b] at (0.5, 1.0) {  }  
child {node [b] at (-0.5, 1.0) {  }  
}
child {node [b] at (0.5, 1.0) {  }  
child {node [b] at (0.0, 1.0) {  }  
}
}
}
;
}}

\newcommand{\forestSE}{
\tikz[planar forest ] {

\node [b] at (0.0, 0.0) {  } 
child {node [b] at (-0.5, 1.0) {  }  
}
child {node [b] at (0.5, 1.0) {  }  
child {node [b] at (-0.5, 1.0) {  }  
child {node [b] at (0.0, 1.0) {  }  
}
}
child {node [b] at (0.5, 1.0) {  }  
}
}
;
}}

\newcommand{\forestTE}{
\tikz[planar forest ] {

\node [a] at (0.0, 0.0) {  } 
;
}}

\newcommand{\forestUE}{
\tikz[planar forest ] {

\node [b] at (0.0, 0.0) {  } 
;
}}

\newcommand{\forestVE}{
\tikz[planar forest ] {

\draw (-2.0,1.0) edge[-] (0.0,1.0);
\draw (0.0,1.0) edge[-] (2.5,1.0);
\draw (-2.0,1.0) arc (90:270:0.5);
\draw (2.5,1.0) arc (90:-90:0.5);
\draw (-2.0,0.0) edge[-] (2.5,0.0);
\node [draw=none] at (0.0, 0.0) { } 
child {node [b,label={ [label distance=-1mm]0:{ \scriptsize a } }] at (-2.0, 1.0) {  } edge from parent[draw=none] 
child {node [b,label={ [label distance=-1mm]0:{ \scriptsize d } }] at (0.0, 1.0) {  }  
}
}
child {node [b,label={ [label distance=-1mm]0:{ \scriptsize b } }] at (0.0, 1.0) {  } edge from parent[draw=none] 
child {node [a,label={ [label distance=-1mm]0:{ \scriptsize e } }] at (-1.0, 1.0) {  }  
child {node [b,label={ [label distance=-1mm]0:{ \scriptsize f } }] at (-0.5, 1.0) {  }  
}
child {node [b,label={ [label distance=-1mm]0:{ \scriptsize g } }] at (0.5, 1.0) {  }  
}
}
child {node [b,label={ [label distance=-1mm]0:{ \scriptsize h } }] at (0.0, 1.0) {  }  
}
child {node [b,label={ [label distance=-1mm]0:{ \scriptsize i } }] at (1.0, 1.0) {  }  
child {node [b,label={ [label distance=-1mm]0:{ \scriptsize j } }] at (0.0, 1.0) {  }  
}
}
}
child {node [b,label={ [label distance=-1mm]0:{ \scriptsize c } }] at (2.5, 1.0) {  } edge from parent[draw=none] 
child {node [b,label={ [label distance=-1mm]0:{ \scriptsize k } }] at (0.0, 1.0) {  }  
child {node [b,label={ [label distance=-1mm]0:{ \scriptsize l } }] at (-0.5, 1.0) {  }  
}
child {node [b,label={ [label distance=-1mm]0:{ \scriptsize m } }] at (0.5, 1.0) {  }  
}
}
}
;
}}

\newcommand{\forestWE}{
\tikz[planar forest ] {

\node [a] at (0.0, 0.0) {  } 
;
}}

\newcommand{\forestXE}{
\tikz[planar forest ] {

\draw (-0.5,1.0) edge[-] (0.5,1.0);
\draw (-0.5,1.0) arc (90:270:0.5);
\draw (0.5,1.0) arc (90:-90:0.5);
\draw (-0.5,0.0) edge[-] (0.5,0.0);
\node [draw=none] at (0.0, 0.0) { } 
child {node [b] at (-0.5, 1.0) {  } edge from parent[draw=none] 
}
child {node [b] at (0.5, 1.0) {  } edge from parent[draw=none] 
}
;
}}

\newcommand{\forestYE}{
\tikz[planar forest ] {

\node [b] at (0.0, 0.0) {  } 
child {node [b] at (-0.5, 1.0) {  }  
}
child {node [b] at (0.5, 1.0) {  }  
}
;
}}

\newcommand{\forestAF}{
\tikz[planar forest ] {

\draw (-0.5,1.0) edge[-] (0.5,1.0);
\draw (-0.5,1.0) arc (90:270:0.5);
\draw (0.5,1.0) arc (90:-90:0.5);
\draw (-0.5,0.0) edge[-] (0.5,0.0);
\node [draw=none] at (0.0, 0.0) { } 
child {node [b] at (-0.5, 1.0) {  } edge from parent[draw=none] 
}
child {node [b] at (0.5, 1.0) {  } edge from parent[draw=none] 
child {node [a] at (-0.5, 1.0) {  }  
}
child {node [b] at (0.5, 1.0) {  }  
}
}
;
}}

\newcommand{\forestBF}{
\tikz[planar forest ] {

\draw (0.0,1.0) arc (90:270:0.5);
\draw (0.0,1.0) arc (90:-90:0.5);
\node [draw=none] at (0.0, 0.0) { } 
child {node [b] at (0.0, 1.0) {  } edge from parent[draw=none] 
}
;
}}

\newcommand{\forestCF}{
\tikz[planar forest ] {

\draw (-0.5,1.0) edge[-] (0.5,1.0);
\draw (-0.5,1.0) arc (90:270:0.5);
\draw (0.5,1.0) arc (90:-90:0.5);
\draw (-0.5,0.0) edge[-] (0.5,0.0);
\node [draw=none] at (0.0, 0.0) { } 
child {node [b] at (-0.5, 1.0) {  } edge from parent[draw=none] 
}
child {node [b] at (0.5, 1.0) {  } edge from parent[draw=none] 
}
;
}}

\newcommand{\forestDF}{
\tikz[planar forest ] {

\draw (-0.75,1.0) edge[-] (0.75,1.0);
\draw (-0.75,1.0) arc (90:270:0.5);
\draw (0.75,1.0) arc (90:-90:0.5);
\draw (-0.75,0.0) edge[-] (0.75,0.0);
\node [draw=none] at (0.0, 0.0) { } 
child {node [b] at (-0.75, 1.0) {  } edge from parent[draw=none] 
child {node [a] at (0.0, 1.0) {  }  
child {node [b] at (-0.5, 1.0) {  }  
}
child {node [b] at (0.5, 1.0) {  }  
}
}
}
child {node [b] at (0.75, 1.0) {  } edge from parent[draw=none] 
child {node [a] at (-0.5, 1.0) {  }  
}
child {node [b] at (0.5, 1.0) {  }  
}
}
;
}}

\newcommand{\forestEF}{
\tikz[planar forest ] {

\draw (0.0,1.0) arc (90:270:0.5);
\draw (0.0,1.0) arc (90:-90:0.5);
\node [draw=none] at (0.0, 0.0) { } 
child {node [b] at (0.0, 1.0) {  } edge from parent[draw=none] 
}
;
}}

\newcommand{\forestFF}{
\tikz[planar forest ] {

\draw (-0.5,1.0) edge[-] (0.5,1.0);
\draw (-0.5,1.0) arc (90:270:0.5);
\draw (0.5,1.0) arc (90:-90:0.5);
\draw (-0.5,0.0) edge[-] (0.5,0.0);
\node [draw=none] at (0.0, 0.0) { } 
child {node [b] at (-0.5, 1.0) {  } edge from parent[draw=none] 
}
child {node [b] at (0.5, 1.0) {  } edge from parent[draw=none] 
}
;
}}

\newcommand{\forestGF}{
\tikz[planar forest ] {

\draw (-0.5,1.0) edge[-] (0.5,1.0);
\draw (-0.5,1.0) arc (90:270:0.5);
\draw (0.5,1.0) arc (90:-90:0.5);
\draw (-0.5,0.0) edge[-] (0.5,0.0);
\node [draw=none] at (0.0, 0.0) { } 
child {node [b] at (-0.5, 1.0) {  } edge from parent[draw=none] 
}
child {node [b] at (0.5, 1.0) {  } edge from parent[draw=none] 
child {node [a] at (-1.0, 1.0) {  }  
child {node [b] at (-0.5, 1.0) {  }  
}
child {node [b] at (0.5, 1.0) {  }  
}
}
child {node [a] at (0.0, 1.0) {  }  
}
child {node [b] at (1.0, 1.0) {  }  
}
}
;
}}

\newcommand{\forestHF}{
\tikz[planar forest ] {

\draw (0.0,1.0) arc (90:270:0.5);
\draw (0.0,1.0) arc (90:-90:0.5);
\node [draw=none] at (0.0, 0.0) { } 
child {node [b] at (0.0, 1.0) {  } edge from parent[draw=none] 
}
;
}}

\newcommand{\forestIF}{
\tikz[planar forest ] {

\draw (-0.5,1.0) edge[-] (0.5,1.0);
\draw (-0.5,1.0) arc (90:270:0.5);
\draw (0.5,1.0) arc (90:-90:0.5);
\draw (-0.5,0.0) edge[-] (0.5,0.0);
\node [draw=none] at (0.0, 0.0) { } 
child {node [b] at (-0.5, 1.0) {  } edge from parent[draw=none] 
}
child {node [b] at (0.5, 1.0) {  } edge from parent[draw=none] 
}
;
}}

\newcommand{\forestJF}{
\tikz[planar forest ] {

\draw (-0.5,1.0) edge[-] (0.5,1.0);
\draw (-0.5,1.0) arc (90:270:0.5);
\draw (0.5,1.0) arc (90:-90:0.5);
\draw (-0.5,0.0) edge[-] (0.5,0.0);
\node [draw=none] at (0.0, 0.0) { } 
child {node [b] at (-0.5, 1.0) {  } edge from parent[draw=none] 
}
child {node [b] at (0.5, 1.0) {  } edge from parent[draw=none] 
child {node [a] at (-0.5, 1.0) {  }  
child {node [b] at (0.0, 1.0) {  }  
child {node [b] at (-0.5, 1.0) {  }  
}
child {node [b] at (0.5, 1.0) {  }  
}
}
}
child {node [b] at (0.5, 1.0) {  }  
}
}
;
}}

\newcommand{\forestKF}{
\tikz[planar forest ] {

\draw (0.0,1.0) arc (90:270:0.5);
\draw (0.0,1.0) arc (90:-90:0.5);
\node [draw=none] at (0.0, 0.0) { } 
child {node [b] at (0.0, 1.0) {  } edge from parent[draw=none] 
}
;
}}

\newcommand{\forestLF}{
\tikz[planar forest ] {

\draw (-0.5,1.0) edge[-] (0.5,1.0);
\draw (-0.5,1.0) arc (90:270:0.5);
\draw (0.5,1.0) arc (90:-90:0.5);
\draw (-0.5,0.0) edge[-] (0.5,0.0);
\node [draw=none] at (0.0, 0.0) { } 
child {node [b] at (-0.5, 1.0) {  } edge from parent[draw=none] 
}
child {node [b] at (0.5, 1.0) {  } edge from parent[draw=none] 
}
;
}}

\newcommand{\forestMF}{
\tikz[planar forest ] {

\draw (-0.5,1.0) edge[-] (0.5,1.0);
\draw (-0.5,1.0) arc (90:270:0.5);
\draw (0.5,1.0) arc (90:-90:0.5);
\draw (-0.5,0.0) edge[-] (0.5,0.0);
\node [draw=none] at (0.0, 0.0) { } 
child {node [b] at (-0.5, 1.0) {  } edge from parent[draw=none] 
}
child {node [b] at (0.5, 1.0) {  } edge from parent[draw=none] 
child {node [a] at (-0.5, 1.0) {  }  
}
child {node [b] at (0.5, 1.0) {  }  
child {node [b] at (0.0, 1.0) {  }  
child {node [b] at (-0.5, 1.0) {  }  
}
child {node [b] at (0.5, 1.0) {  }  
}
}
}
}
;
}}

\newcommand{\forestNF}{
\tikz[planar forest ] {

\draw (0.0,1.0) arc (90:270:0.5);
\draw (0.0,1.0) arc (90:-90:0.5);
\node [draw=none] at (0.0, 0.0) { } 
child {node [b] at (0.0, 1.0) {  } edge from parent[draw=none] 
}
;
}}

\newcommand{\forestOF}{
\tikz[planar forest ] {

\draw (-0.5,1.0) edge[-] (0.5,1.0);
\draw (-0.5,1.0) arc (90:270:0.5);
\draw (0.5,1.0) arc (90:-90:0.5);
\draw (-0.5,0.0) edge[-] (0.5,0.0);
\node [draw=none] at (0.0, 0.0) { } 
child {node [b] at (-0.5, 1.0) {  } edge from parent[draw=none] 
}
child {node [b] at (0.5, 1.0) {  } edge from parent[draw=none] 
}
;
}}

\newcommand{\forestPF}{
\tikz[planar forest ] {

\draw (-0.5,1.0) edge[-] (0.5,1.0);
\draw (-0.5,1.0) arc (90:270:0.5);
\draw (0.5,1.0) arc (90:-90:0.5);
\draw (-0.5,0.0) edge[-] (0.5,0.0);
\node [draw=none] at (0.0, 0.0) { } 
child {node [b] at (-0.5, 1.0) {  } edge from parent[draw=none] 
}
child {node [b] at (0.5, 1.0) {  } edge from parent[draw=none] 
child {node [a] at (-0.5, 1.0) {  }  
}
child {node [b] at (0.5, 1.0) {  }  
}
}
;
}}

\newcommand{\forestQF}{
\tikz[planar forest ] {

\draw (0.0,1.0) arc (90:270:0.5);
\draw (0.0,1.0) arc (90:-90:0.5);
\node [draw=none] at (0.0, 0.0) { } 
child {node [b] at (0.0, 1.0) {  } edge from parent[draw=none] 
child {node [a] at (0.0, 1.0) {  }  
child {node [b] at (-0.5, 1.0) {  }  
}
child {node [b] at (0.5, 1.0) {  }  
}
}
}
;
}}

\newcommand{\forestRF}{
\tikz[planar forest ] {

\draw (0.0,1.0) arc (90:270:0.5);
\draw (0.0,1.0) arc (90:-90:0.5);
\node [draw=none] at (0.0, 0.0) { } 
child {node [b] at (0.0, 1.0) {  } edge from parent[draw=none] 
child {node [b] at (0.0, 1.0) {  }  
}
}
;
}}

\newcommand{\forestSF}{
\tikz[planar forest ] {

\node [b] at (0.0, 0.0) {  } 
child {node [b] at (-0.5, 1.0) {  }  
}
child {node [b] at (0.5, 1.0) {  }  
}
;
}}

\newcommand{\forestTF}{
\tikz[planar forest ] {

\node [b] at (0.0, 0.0) {  } 
child {node [b,label={ [label distance=-1mm]0:{ \scriptsize v } }] at (-0.5, 1.0) {  }  
child {node [b] at (-0.5, 1.0) {  }  
}
child {node [b] at (0.5, 1.0) {  }  
}
}
child {node [b] at (0.5, 1.0) {  }  
}
;
}}

\newcommand{\forestUF}{
\tikz[planar forest ] {

\draw (0.0,1.0) arc (90:270:0.5);
\draw (0.0,1.0) arc (90:-90:0.5);
\node [draw=none] at (0.0, 0.0) { } 
child {node [b] at (0.0, 1.0) {  } edge from parent[draw=none] 
child {node [b] at (0.0, 1.0) {  }  
}
}
;
}}

\newcommand{\forestVF}{
\tikz[planar forest ] {

\node [b] at (0.0, 0.0) {  } 
child {node [b,label={ [label distance=-1mm]0:{ \scriptsize v } }] at (-0.5, 1.0) {  }  
child {node [b] at (-1.0, 1.0) {  }  
}
child {node [b] at (0.0, 1.0) {  }  
child {node [b] at (-0.5, 1.0) {  }  
}
child {node [b] at (0.5, 1.0) {  }  
}
}
child {node [b] at (1.0, 1.0) {  }  
}
}
child {node [b] at (0.5, 1.0) {  }  
}
;
}}

\newcommand{\forestWF}{
\tikz[planar forest ] {

\draw (-0.5,1.0) edge[-] (0.5,1.0);
\draw (-0.5,1.0) arc (90:270:0.5);
\draw (0.5,1.0) arc (90:-90:0.5);
\draw (-0.5,0.0) edge[-] (0.5,0.0);
\node [draw=none] at (0.0, 0.0) { } 
child {node [b] at (-0.5, 1.0) {  } edge from parent[draw=none] 
}
child {node [b] at (0.5, 1.0) {  } edge from parent[draw=none] 
child {node [b] at (0.0, 1.0) {  }  
}
}
;
}}

\newcommand{\forestXF}{
\tikz[planar forest ] {

\node [b,label={ [label distance=-1mm]0:{ \scriptsize v } }] at (0.0, 0.0) {  } 
child {node [b] at (-1.0, 1.0) {  }  
}
child {node [b] at (0.0, 1.0) {  }  
}
child {node [b] at (1.0, 1.0) {  }  
child {node [b] at (0.0, 1.0) {  }  
}
}
;
}}

\newcommand{\forestYF}{
\tikz[planar forest ] {

\draw (-0.5,1.0) edge[-] (0.5,1.0);
\draw (-0.5,1.0) arc (90:270:0.5);
\draw (0.5,1.0) arc (90:-90:0.5);
\draw (-0.5,0.0) edge[-] (0.5,0.0);
\node [draw=none] at (0.0, 0.0) { } 
child {node [b] at (-0.5, 1.0) {  } edge from parent[draw=none] 
}
child {node [b] at (0.5, 1.0) {  } edge from parent[draw=none] 
child {node [b] at (0.0, 1.0) {  }  
}
}
;
}}

\newcommand{\forestAG}{
\tikz[planar forest ] {

\draw (0.0,1.0) arc (90:270:0.5);
\draw (0.0,1.0) arc (90:-90:0.5);
\node [draw=none] at (0.0, 0.0) { } 
child {node [b] at (0.0, 1.0) {  } edge from parent[draw=none] 
child {node [a] at (-1.0, 1.0) {  }  
}
child {node [a] at (0.0, 1.0) {  }  
}
child {node [b] at (1.0, 1.0) {  }  
child {node [b] at (0.0, 1.0) {  }  
}
}
}
;
}}

\newcommand{\forestBG}{
\tikz[planar forest ] {

\draw (-1.5,1.0) edge[-] (0.0,1.0);
\draw (0.0,1.0) edge[-] (1.0,1.0);
\draw (-1.5,1.0) arc (90:270:0.5);
\draw (1.0,1.0) arc (90:-90:0.5);
\draw (-1.5,0.0) edge[-] (1.0,0.0);
\node [draw=none] at (0.0, 0.0) { } 
child {node [b] at (-1.5, 1.0) {  } edge from parent[draw=none] 
child {node [b] at (0.0, 1.0) {  }  
}
}
child {node [b,label={ [label distance=-1mm]0:{ \scriptsize v } }] at (0.0, 1.0) {  } edge from parent[draw=none] 
child {node [b] at (-0.5, 1.0) {  }  
}
child {node [b] at (0.5, 1.0) {  }  
}
}
child {node [b] at (1.0, 1.0) {  } edge from parent[draw=none] 
}
;
}}

\newcommand{\forestCG}{
\tikz[planar forest ] {

\node [b] at (0.0, 0.0) {  } 
child {node [b] at (0.0, 1.0) {  }  
}
;
}}

\newcommand{\forestDG}{
\tikz[planar forest ] {

\draw (-2.0,1.0) edge[-] (0.0,1.0);
\draw (0.0,1.0) edge[-] (1.0,1.0);
\draw (-2.0,1.0) arc (90:270:0.5);
\draw (1.0,1.0) arc (90:-90:0.5);
\draw (-2.0,0.0) edge[-] (1.0,0.0);
\node [draw=none] at (0.0, 0.0) { } 
child {node [b] at (-2.0, 1.0) {  } edge from parent[draw=none] 
child {node [b] at (0.0, 1.0) {  }  
}
}
child {node [b,label={ [label distance=-1mm]0:{ \scriptsize v } }] at (0.0, 1.0) {  } edge from parent[draw=none] 
child {node [b] at (-1.0, 1.0) {  }  
child {node [b] at (0.0, 1.0) {  }  
}
}
child {node [b] at (0.0, 1.0) {  }  
}
child {node [b] at (1.0, 1.0) {  }  
}
}
child {node [b] at (1.0, 1.0) {  } edge from parent[draw=none] 
}
;
}}

\newtheorem{theorem}{Theorem}[section]
\newtheorem{corollary}{Corollary}[theorem]
\newtheorem{lemma}[theorem]{Lemma}
\newtheorem{proposition}[theorem]{Proposition}
\newtheorem{definition}[theorem]{Definition}
\newtheorem{example}[theorem]{Example}

\DeclareMathOperator{\graft}{\triangleright}
\DeclareMathOperator{\Div}{\text{Div}}
\DeclareMathOperator{\bgraft}{\blacktriangleright}

\title{The Universal Post-Lie-Rinehart Algebra of Planar Aromatic Trees}

\author{Ludwig Rahm\footnote{Section de Mathématiques, Université de Genève. \texttt{ludwig.rahm@unige.ch}.}}

\begin{document}
\maketitle

\begin{abstract}
	This paper defines the algebraic structure of tracial post-Lie-Rinehart algebras and describes the free object in this category. Post-Lie-Rinehart algebras is a generalisation of pre-Lie-Rinehart algebras, and of post-Lie algebroids.
\end{abstract}

\section{Introduction}

B-series were introduced by John Butcher \cite{Butcher1963} in the 1960s, as a way to index numerical methods for ODEs by using rooted trees. This indexation is nowadays understood in the context of pre-Lie algebras as a pre-Lie morphism. The universal pre-Lie algebra is given by the vector space spanned by rooted trees together with the grafting product \cite{ChapotonLivernet2000}. We use $\mathcal{T}$ to denote the vector space, and $T$ to denote its canonical basis. The universality property means that any map $\bullet \to x$ sending the single-vertex tree $\bullet$ to some element $x$ in another pre-Lie algebra $\mathcal{A}$, extends uniquely to a pre-Lie morphism $\mathcal{T} \to \mathcal{A}$. A B-series is such a pre-Lie morphism. Every Euclidean space naturally comes equipped with a flat and torsion free connection. Using this connection as a product, the space of vector fields on $\mathbb{R}^d$ becomes a pre-Lie algebra. A B-series $B_a$ is an element in the graded completion of the free pre-Lie algebra, indexed by the coefficient map $a:T \to \mathbb{R}$. By the universality property any map $\bullet \to f$ mapping the single-vertex tree to a vector field $f$, uniquely extends to a pre-Lie morphism $B_a \to B_a(f)$ sending the B-series $B_a$ to an infinite series of vector fields. In applications, the B-series $B_a$ indexes a B-series method while $B_a(f)$ describes the method applied to a specific ODE $y'=f(y)$. The algebraic properties of such series has both been used to study properties of the associated numerical methods, and to study geometric properties of the associated series of vector fields. A connection between B-series and affine equivariant families of mappings of vector fields on Euclidean spaces was established in \cite{MclachlanModinMuntheKaasVerdier2016}. This connection was further studied in \cite{MuntheKaasVerdier2015}, where the authors looked at a generalisation of B-series known as \textit{aromatic B-series} and proved he following important theorem.
\begin{theorem}\cite{MuntheKaasVerdier2015}
	Let $L$ be the canonical pre-Lie algebra of vector fields on a finite-dimensional Euclidean space. A smooth local mapping $\Phi: L \to L$ can be expanded in an aromatic B-series if and only if $\Phi \circ \xi = \xi \circ \Phi$ for all pre-Lie isomorphisms $\xi: L \to L$.
\end{theorem}
Aromatic B-series is a generalisation of B-series introduced in \cite{IserlesQuispelTse2006,MuntheKaasVerdier2015}, motivated by the study of volume preserving integration algorithms. An aroma is a connected directed graph where each node has exactly one outgoing edge. The set of all aromas is denoted $A$, they span the vector space $\mathcal{A}$, and we write $S(\mathcal{A})$ for the free symmetric algebra generated by $\mathcal{A}$. As a vector space, $S(\mathcal{A})$ has a canonical basis consisting of all directed graphs where each node has exactly one outgoing edge. Whereas a B-series $B_a$ is indexed by a coefficient map $a: T \to \mathbb{R}$, an aromatic B-series $AB_{\alpha}$ is indexed by a coefficient map $\alpha: T \to S(\mathcal{A})$. It was shown in \cite{FloystadManchonMuntheKaas2020} that the algebraic structure behind aromatic B-series is \textit{pre-Lie-Rinehart algebras}, and that the space of aromatic trees $\mathcal{AT} := S(\mathcal{A}) \otimes \mathcal{T}$ gives the free tracial pre-Lie-Rinehart algebra. A pre-Lie-Rinehart algebra is a Lie-Rinehart algebra equipped with a flat and torsion free connection.

B-series, pre-Lie algebras and the formalism of rooted trees has been generalised in several other directions, including Lie-Butcher series \cite{Munthe-KaasWright2008}, (planarly) branched rough paths \cite{CurryEbrahimiFardManchonMuntheKaas2018,Gubinelli2010} and regularity structures \cite{Hairer2013}. Lie-Butcher series generalise B-series to homogeneous spaces, and are used for Lie group integration methods. These spaces comes equipped with a flat connection that has constant torsion, and their vector fields form a so-called post-Lie algebra \cite{Munthe-KaasWright2008,Vallette2007} when equipped with this connection. One can now ask if the success of aromatic B-series for Euclidean spaces, can be replicated also on homogeneous spaces. This paper aims to take a first step in this direction by describing the free post-Lie-Rinehart algebra, being the homogeneous space generalisation of the free pre-Lie-Rinehart algebra.

The paper is organized as follows. In section \ref{sec::Prelim}, we recall the necessary background information on pre-Lie-Rinehart algebras and post-Lie algebras. In section \ref{sec::PostLieRinehart} we construct the free Post-Lie-Rinehart algebra.

\section{Preliminaries} \label{sec::Prelim}

We recall the definitions and basic results on pre-Lie Rinehart algebras and on post-Lie algebras.

\subsection{Tracial pre-Lie Rinehart algebras}

All definitions and results in this subsection are taken from \cite{FloystadManchonMuntheKaas2020}. We first recall the definition of a Lie-Rinehart algebra.

\begin{definition}
	Let $R$ be a unital commutative algebra. A Lie-Rinehart algebra over $R$ consists of an $R$-module $L$ and an $R$-linear map
	\begin{align*}
		\rho: L \to \text{Der}(R,R)
	\end{align*}
	such that
	\begin{itemize}
		\item $L$ is a Lie algebra with Lie bracket $\llbracket \cdot,\cdot \rrbracket$.
		\item $\rho$ is a Lie algebra morphism.
		\item For $f \in R$ and $X,Y \in L$, the identity
		\begin{align*}
			\llbracket X,fY \rrbracket=(\rho(X)f)Y + f\llbracket X,Y\rrbracket
		\end{align*}
		holds.
	\end{itemize}
\end{definition}

The map $\rho$ in the definition above is called the \textit{anchor map}. In order to talk about the Lie-Rinehart algebra being pre-Lie, we furthermore need the notion of a \textit{connection}. Note that the following definition agrees with the algebraic properties that define a connection on a manifold.

\begin{definition}
	A connection is an $R$-linear map
	\begin{align*}
		\triangledown:& L \to \text{End}(L),\\
		&X \to \triangledown_X,
	\end{align*}
	such that:
	\begin{align*}
		\triangledown_X(fY)=(\rho(X)f)Y+f\triangledown_XY.
	\end{align*}
\end{definition}

If $(L,R)$ is a Lie-Rinehart algebra with connection $\triangledown$, one can respectively define the torsion $\mathcal{T}$ and curvature $\mathcal{R}$ by:
\begin{align*}
	\mathcal{T}(X,Y):=&\triangledown_XY - \triangledown_YX - \llbracket X,Y \rrbracket, \\
	\mathcal{R}(X,Y,Z):=& \triangledown_X (\triangledown_Y Z) - \triangledown_Y (\triangledown_X Z ) - \triangledown_{\llbracket X,Y \rrbracket } Z.
\end{align*}

Again note that these definitions agrees with the torsion and curvature of a manifold with connection. If the connection is flat and torsion-free, meaning that both $\mathcal{T},\mathcal{R}$ are identically zero, then $(L,\triangledown)$ forms a \textit{pre-Lie algebra}. This means that the connection $\triangledown$, seen as a product on $L$, satisfies the pre-Lie identity:
\begin{align*}
	\triangledown_X (\triangledown_Y Z) - \triangledown_{\triangledown_X Y}Z - \triangledown_Y(\triangledown_X Z) + \triangledown_{\triangledown_Y X}Z =0.
\end{align*}
In the guiding example of a manifold with a connection, the connection is flat and torsion-free exactly when the manifold is locally Euclidean. 

\begin{definition}
	A pre-Lie-Rinehart algebra is a Lie-Rinehart algebra endowed with a flat and torsion-free connection.
\end{definition}

A connection $\triangledown: L \to \text{End}(L)$ extends to a map $\triangledown: L \to \text{End}(\text{Hom}(L))$ by:
\begin{align*}
	\triangledown_X(u)(Y)=\triangledown_X(u(Y))-u(\triangledown_X Y).
\end{align*}
This lets us define the \textit{algebra of elementary R-module homomorphisms}.

\begin{definition}
	Let the linear operator $\delta$ denote the map
	\begin{align*}
		\delta: L \to& \hom(L),\\
		\delta X(Z)=&\triangledown_Z X.
	\end{align*}
	Let the algebra of elementary R-module endomorphisms be the R-module subalgebra of $\hom_R(L)$ generated by $\{\triangledown_{Y_1}\dots \triangledown_{Y_n}\delta X: X,Y_1,\dots,Y_n \in L  \}$. It will be denoted by $E\ell_R(L)$.
\end{definition}

\begin{definition}
	Let $(L,R)$ be a Lie-Rinehart algebra. We say that the Lie-Rinehart algebra is tracial if there exists a map $\tau: E\ell_R(L) \to R$ such that:
	\begin{align*}
		\tau(\alpha \circ \beta)=&\tau(\beta \circ \alpha),\\
		\tau(\triangledown_X \alpha)=&\rho(X)\tau(\alpha).
	\end{align*}
	In this case, we denote $\Div := \tau \circ \delta$.
\end{definition}

\begin{proposition} \cite{FloystadManchonMuntheKaas2020}
	In a pre-Lie-Rinehart algebra $L$, the algebra $E\ell_R(L)$ is generated by $\{\delta X:X \in L \}$.
\end{proposition}

\subsection{The free tracial pre-Lie-Rinehart algebra} \label{ssec::FreePreLieRinehart}

The free tracial pre-Lie-Rinehart algebra was described in \cite{FloystadManchonMuntheKaas2020}. We first introduce some notation. Let $\mathcal{T}$ denote the vector space spanned by non-planar rooted trees and let $\mathcal{F}$ denote the vector space of forests, meaning commutative words in non-planar rooted trees.
\begin{align*}
	\mathcal{T}:=\text{span}\{\forestA,\forestB,\forestC,\forestD,\forestE,\forestF=\forestG,\forestH,\forestI   \dots  \}.
\end{align*}
Letting $\mathcal{C}$ be some set, one can then define $\mathcal{T}_{\mathcal{C}}$ as the vector space of rooted trees where each vertex is labeled by an element from $\mathcal{C}$. The following constructions would then lead to the free objects generated by the set $\mathcal{C}$, however we will leave our vertices (mostly) unlabeled and only explicitly describe the single-generator case for convenience. We may in the future label our vertices as a way to refer to a specific vertex, this does not mean that the label is part of a generating set.

An aroma is a connected directed graph where every vertex has exactly one incoming edge, the vector space spanned by aromas is denoted $\mathcal{A}$. An aroma consists of a central cycle, with trees attached to each vertex in this cycle. Note that the edges are oriented away from the cycle, and note that the cycle could consist of a single root with a loop edge. We will represent aromas by this central cycle, with edges going right-to-left, note that the representation is invariant under cyclic permutations:
\begin{align*}
	\forestJ=\forestK=\forestL.
\end{align*}
The symmetric algebra generated by $\mathcal{A}$ is denoted $S(\mathcal{A})$. We denote the space $S(\mathcal{A})\otimes \mathcal{T}$ by $\mathcal{AT}$, these are called \textit{aromatic trees}. The space $\mathcal{AF}=S(\mathcal{A})\otimes S(\mathcal{T})$ are called the aromatic forests, these are represented by commutative words where the letters are aromas and trees. We shall use the non-caliographic $A,T,F,AT,AF$ to denote sets rather than vector spaces.

We define the product $\curvearrowright$ on $\mathcal{AT}$ by letting $t_1 \curvearrowright t_2$ be the sum of all aromatic trees obtained by adding an edge from some vertex of $t_2$ to the root of $t_1$. For example:
\begin{align*}
	\forestM \curvearrowright \forestN=& \forestO + \forestP \\
	\forestQ \curvearrowright \forestR\forestS=& \forestT\forestU+\forestV\forestW + \forestX\forestY \\
	\forestAB\forestBB \curvearrowright \forestCB\forestDB=&\forestEB\forestFB\forestGB+\forestHB\forestIB\forestJB + \forestKB \forestLB\forestMB.
\end{align*}

This product then lets us define the following maps, which turns $(\mathcal{AT},S(\mathcal{\mathcal{A}}))$ into a pre-Lie-Rinehart algebra:
\begin{align*}
	\triangledown:& \mathcal{AT} \to \text{Hom}(\mathcal{AT},\mathcal{AT}), \; t \to (s \to t \curvearrowright s),\\
	\rho:& \mathcal{AT} \to \text{Der}(S(\mathcal{A}),S(\mathcal{A})), \; t \to (s \to t \curvearrowright s),\\
	\delta:& \mathcal{AT}\to \text{Hom}_{S(\mathcal{A})}(\mathcal{AT},\mathcal{AT}), \; t \to (s \to s \curvearrowright t).
\end{align*}

\begin{proposition} \cite{FloystadManchonMuntheKaas2020}
	$\mathcal{AT}$ is a pre-Lie-Rinehart algebra over $S(\mathcal{A})$ with anchor map $\rho$ and connection $\triangledown$.
\end{proposition}

Let $\mathcal{V}$ denote the vector space of pairs $(v,t)$ where $t \in T$ and $v$ is a vertex of $t$. We let elements of $\mathcal{V}$ represent endomorphisms of $\mathcal{T}$, the element $(v,t)$ represents the map $s \to s \curvearrowright_v t$, which sends $s$ to the aromatic tree obtained by adding an edge from the vertex $v$ to the root of $s$. For example:
\begin{align*}
	\forestNB\forestOB \curvearrowright_v \forestPB\forestQB=\forestRB\forestSB\forestTB.
\end{align*}
Let $\delta: \mathcal{T} \to \mathcal{V}$ denote the map that sends $t$ to $\sum_{v \in t} (v,t)$, and let $\circ$ denote the product on $\mathcal{V}$ given by $(v,t)\circ (v',t')=(v,t\curvearrowright_{v'}t')$. Next define the map $\tau: \mathcal{V} \to \mathcal{A}$ by sending $(v,t)$ to the aroma obtained by adding an edge from the vertex $v$ to the root of $t$. For example:
\begin{align*}
\tau((d, \forestUB) )=\forestVB.
\end{align*}
Note that composition in $\mathcal{V}$ is an associative product, hence the commutator 
\begin{align*}
	[(v_1,t_1),(v_2,t_2) ]:=& (v_1,t_1)\circ (v_2,t_2) - (v_2,t_2)\circ(v_1,t_1) \\
	=&(v_1,t_1 \curvearrowright_{v_2} t_2) - (v_2,t_2 \curvearrowright_{v_2} t_1)
\end{align*}
is a Lie bracket. Then we have the following lemma:

\begin{lemma}\cite{FloystadManchonMuntheKaas2020}
	The vector space $\mathcal{A}$ can naturally be identified with the quotient space $\mathcal{V}/[\mathcal{V},\mathcal{V}]$, so that $\tau$ becomes the natural projection from $\mathcal{V}$ to $\mathcal{V}/[\mathcal{V},\mathcal{V}]$.
\end{lemma}

It will in the future be advantageous to represent an aroma in $\mathcal{A}$ by choosing a representative from the pre-image of $\tau$ in $\mathcal{V}$. \\

Now let $\mathcal{E}$ denote the vector space of pairs $(v,A_1\dots A_n t)$ where $A_1 \dots A_n t \in \mathcal{AT}$ and $v$ is a vertex in $A_1 \dots A_n t$. We let elements of $\mathcal{E}$ represent $S(\mathcal{A})$-linear endomorphisms of $\mathcal{AT}$, by letting $(v,A_1\dots A_n t)$ represent the map $A'_1\dots A'_m t' \to (A'_1 \dots A'_m t') \curvearrowright_v (A_1 \dots A_n t)$. We extend the map $\delta$ to be a map $\mathcal{AT} \to S(\mathcal{A})$. Furthermore let $\mathcal{B}$ denote the vector space of aromas with one marked vertex, $(v,A_1)$ where $A_1\in A$ and $v$ is a vertex in $A_1$. These represent homomorphisms $\mathcal{T} \to \mathcal{A}$ by $t \to (t \curvearrowright_v A)$. Extend $\delta$ to a map $\mathcal{A} \to \mathcal{B}$, again by summing over all vertices. Now let $\mathcal{D}$ denote the vector space generated by expressions $(\delta A)t$, where $A$ is an aroma and $t$ is a tree. We have natural compositions:
\begin{align*}
	D_c \circ V_c :& (\delta A_1 t_1 ) \circ (v,t_2) = (v,t_2 \curvearrowright A_1)t_1,\\
	V_c \circ D_c:& (v,t_1) \circ (\delta A_1 t_2)=\delta A_1(t_2 \curvearrowright_v t_1 ),\\
	D_c \circ D_c:& (\delta A_1 t_1) \circ (\delta A_2 t_2) = (t_2 \curvearrowright A_1) \delta A_2 t_1.
\end{align*}

This lets us represent the algebra of elementary endomorphisms as as a direct sum between aromatic trees with a marked vertex on a tree, aromatic trees with a marked vertex in the interior cycle of an aroma, and aromatic trees with a marked vertex on an aroma but not in the interior cycle:

\begin{proposition} \cite{FloystadManchonMuntheKaas2020}
	The algebra of elementary endomorphisms $E\ell_{S(\mathcal{A})}(\mathcal{AT})$ is generated by $\mathcal{D},\mathcal{V}$ and the space spanned by all of their compositions $\mathcal{D} \circ \mathcal{V}$. It decomposes as a free $S(\mathcal{A})$-module:
	\begin{align*}
		E\ell_{S(\mathcal{A})}(\mathcal{AT})= (S(\mathcal{A})\otimes \mathcal{V}) \bigoplus (S(\mathcal{A}) \otimes \mathcal{D} ) \bigoplus (S(\mathcal{A}) \otimes (\mathcal{D} \circ \mathcal{V} )  ).
	\end{align*}
\end{proposition}

\begin{corollary} \cite{FloystadManchonMuntheKaas2020}
	The pre-Lie-Rinehart algebra $\mathcal{AT}$ is tracial.
\end{corollary}

\begin{theorem} \cite{FloystadManchonMuntheKaas2020}
	The pre-Lie-Rinehart algebra $(\mathcal{AT},S(\mathcal{A}))$ is the universal tracial pre-Lie-Rinehart algebra.
\end{theorem}

The pre-Lie algebra $(\mathcal{AT},S(\mathcal{A}))$ being universal does in particular mean that, starting from the single-vertex trees, one can generate all of $\mathcal{AT}$ and all of $S(\mathcal{A})$ by using the pre-Lie grafting product $\curvearrowright$, the anchor map $\rho$, the map $\delta$, composition $\circ$ in  $E\ell_{S(\mathcal{A})}(\mathcal{AT})$, the trace map $\tau$ and the module action of $S(\mathcal{A})$ onto $\mathcal{AT}$ and $E\ell_{S(\mathcal{A})}(\mathcal{AT})$. The construction for how to generate $\mathcal{AT}$ and all of $S(\mathcal{A})$ using these operations can be found in \cite{FloystadManchonMuntheKaas2020}, although spread out across multiple proofs. We will recall the construction here. We build up the spaces in a few steps:
\begin{enumerate}
	\item The space $\mathcal{T}$ is the free pre-Lie algebra, hence we can generate it by using the single-vertex trees and the product $\curvearrowright$.
	\item By using $\mathcal{T}$, we can generate $\mathcal{V}$. This is described in the proof of \cite[Lemma 3.4]{FloystadManchonMuntheKaas2020}, by an induction over the number of vertices. As a base case, the single-vertex tree $(d,\forestWB)$ is the image of $\forestXB$ under the map $\delta$. Assume that we can generate all trees with at most $N$ vertices, and let $(v,t) \in \mathcal{V}$ have $N+1$ vertices. If $v$ is not the root of $t$, then it is attached to a vertex $w$. Let $t_v$ be the subtree of $t$ that has $v$ as a root, and let $t'$ be the tree obtained by removing $t_v$ from $t$. Then $(v,t)=(w,t')\circ (v,t_v)$ and $(v,t)$ can be generated by induction. If $v$ is the root of $t$, then $(v,t)=\delta(t) - \sum_{w \neq v} (w,t)$ and we can generate the right side of the equation, hence also $(v,t)$.
	\item From having generated all of $\mathcal{V}$, we can generate all of $\mathcal{A}$ as the image of $\mathcal{V}$ under the map $\tau$.
	\item By using the module action of $S(\mathcal{A})$, we can now generate all aromatic trees with at most one aroma as well as $\mathcal{A} \otimes \mathcal{V}$.
	\item Let $\delta(A)t \in \mathcal{D}$ be arbitrary, then $\delta(A)t= \delta(At)-A\delta(t)$. We can generate the terms in the right side of the equation, hence also all of $\mathcal{D}$.
	\item If we can generate all of $\mathcal{V}$ and all of $\mathcal{D}$, we can compose to generate all of $\mathcal{D} \circ \mathcal{V}$.
	\item Because we can generate all of $\mathcal{A}$, we can now inductively use the module action to build aromatic trees and multi-aromas with increasing number of aromas and hence generate the full spaces.
\end{enumerate}

Having abstractly described the construction, we now want to illustrate it with an example.

\begin{example} \label{example::AromaticTreeGeneration}
	We want to describe the aromatic tree
	\begin{align*}
		A_1 A_2 t = \forestYB\forestAC\forestBC
	\end{align*}
	in terms of generators and operators. As a first step, we separate the aromatic tree into a tree-part and a multi-aroma part. $A_1A_2t$ is the module action of the multi-aroma $A_1A_2$ onto the tree $t$. The tree can be written as
	\begin{align*}
		t= \forestCC \curvearrowright \forestDC.
	\end{align*}
	We then need to describe the multi-aroma $A_1A_2$ in terms of generators and operators. For this, we need to express $A_1A_2$ as the image of something in $E\ell_{S(\mathcal{A})}(\mathcal{AT})$ under the map $\tau$. As $\tau$ is not injective, this choice is not unique. Due to the universality property, different choices here will lead to different but equal expressions in terms of the generators. We can choose any edge in $A_1A_2$ to be the edge added by the map $\tau$, we choose the edge from $\forestEC$ to itself. Then we can write
	\begin{align*}
		A_1A_2 =& \tau\bigl( \forestFC (f,\forestGC) \bigr)\\
		=& \tau\bigl( \forestHC \delta(\forestIC) \bigr). 
	\end{align*}
	It then remains to express the aroma $A_1$ in terms of generators. We have to express $A_1$ as the image of something under the map $\tau$, and we choose the edge between vertices $c$ and $a$ to be added by this map, so that:
	\begin{align*}
		A_1 = \tau\bigl(  (a, \forestJC )  \bigr).
	\end{align*}
	We then have to use the induction from point $2$ above to express $(a, \forestKC)$ in terms of generators:
	\begin{align*}
		(a, \forestLC ) =& (a,\forestMC) \circ (b,\forestNC)\\
		=&(a,\forestOC) \circ (b,\forestPC) \circ (c,\forestQC)\\
		=&\bigl(\delta(\forestRC) - (e,\forestSC)  \bigr) \circ \bigl(\delta(\forestTC) - (e,\forestUC)  \bigr) \circ \delta(\forestVC) \\
		=& \bigl(\delta(\forestWC) - (d,\forestXC) \circ  (a,\forestYC) \bigr) \circ \bigl(\delta(\forestAD) - (e,\forestBD)\circ (b,\forestCD)  \bigr) \circ \delta(\forestDD) \\
		=&\bigl(  \delta(\forestED\curvearrowright \forestFD ) - \delta(\forestGD) \circ \delta(\forestHD) \bigr) \circ \bigl(  \delta(\forestID\curvearrowright \forestJD ) - \delta(\forestKD) \circ \delta(\forestLD) \bigr) \circ \delta(\forestMD).
	\end{align*}
	In total we get:
	\begin{align*}
		A_1A_2t = \tau\Biggl( \tau \Bigl( \bigl(  \delta(\forestND\curvearrowright \forestOD ) - \delta(\forestPD) \circ \delta(\forestQD) \bigr) \circ \bigl(  \delta(\forestRD\curvearrowright \forestSD ) - \delta(\forestTD) \circ \delta(\forestUD) \bigr) \circ \delta(\forestVD) \Bigr) \delta(\forestWD)  \Biggr)(\forestXD \curvearrowright \forestYD).
	\end{align*}
\end{example}

\subsection{Post-Lie algebras}

A post-Lie algebra $(\mathfrak{g},[\cdot,\cdot],\graft)$ is a Lie algebra $(\mathfrak{g},[\cdot,\cdot])$ together with a second product $\graft: \mathfrak{g} \otimes \mathfrak{g} \to \mathfrak{g}$ satisfying the two following conditions for all $X,Y,Z \in \mathfrak{g}$:
\begin{align}
	X \graft [Y,Z]=& [X \graft Y,Z] + [Y,X \graft Z],\label{eq::PostLie1} \\
	[X,Y] \graft Z =& X \graft (Y \graft Z) - (X \graft Y)\graft Z - Y \graft (X \graft Z) + (Y \graft X) \graft Z. \label{eq::PostLie2}
\end{align}
A post-Lie algebra with abelian Lie bracket is pre-Lie, as one recovers the pre-Lie identity by setting the Lie brackets to zero in the above identities. One of the motivating examples for the definition of post-Lie algebras, is the space of vector fields on a homogeneous space. Let $\mathcal{X}(M)$ denote the space of vector fields on the manifold $M$, equipped with a connection $\triangledown$. The connection $\triangledown$ defines a product $\graft$ on $\mathcal{X}(M)$ by $X \graft Y := \triangledown_X Y$. Recall the torsion $\mathcal{T}$ and the curvature $\mathcal{R}$ given by: 
\begin{align}
	\mathcal{T}(X,Y):=&\triangledown_XY - \triangledown_YX - \llbracket X,Y \rrbracket, \label{eq::Torsion} \\
	\mathcal{R}(X,Y,Z):=& \triangledown_X (\triangledown_Y Z) - \triangledown_Y (\triangledown_X Z ) - \triangledown_{\llbracket X,Y \rrbracket } Z, \label{eq::Curvature}
\end{align}
where the Lie bracket $\llbracket \cdot , \cdot \rrbracket$ is the Jacobi bracket of vector fields. In the case that the connection is flat and has constant torsion, meaning $\mathcal{R} = 0 = \triangledown \mathcal{T}$, then $(\mathcal{X}(M),-\mathcal{T}(\cdot,\cdot), \graft )$ is a post-Lie algebra.

A post-Lie algebra always comes equipped with two Lie brackets. Let $(\mathfrak{g},[\cdot,\cdot],\graft)$ be a post-Lie algebra, then the product $\llbracket X,Y \rrbracket := X \graft Y - Y \graft X + [X,Y]$ is always a second Lie bracket. In the homogeneous space example above, this is the relation between torsion and the Jacobi bracket. One could in fact view a post-Lie algebra as a structure defined a relation between two different Lie algebras over the same space. Indeed instead of starting with $(\mathfrak{g},[\cdot,\cdot],\graft)$ and requiring the identities \eqref{eq::PostLie1}-\eqref{eq::PostLie2}, one could equivalently start with the triple $(\mathfrak{g},\llbracket\cdot,\cdot \rrbracket,\graft)$ then define torsion and curvature via equations \eqref{eq::Torsion}-\eqref{eq::Curvature}. One then ends up with a post-Lie algebra $(\mathfrak{g},-\mathcal{T}(\cdot,\cdot)=[\cdot,\cdot],\graft)$ if and only if the curvature is zero and the torsion is constant.

Post-Lie algebras always comes in pairs. Let $(\mathfrak{g},[\cdot,\cdot],\graft)$ be a post-Lie algebra and define the \textit{sub-adjacent} post-Lie product $X \bgraft Y := X \graft Y + [X,Y]$ on $\mathfrak{g}$, then $(\mathfrak{g},-[\cdot,\cdot],\bgraft)$ is also a post-Lie algebra. The two post-Lie algebras generate the same second Lie bracket $\llbracket \cdot,\cdot \rrbracket$.

\subsection{The free post-Lie algebra}

Let $PT$ denote the set of \textit{planar rooted trees}. These are rooted trees endowed with an embedding into the plane, meaning that for each vertex there is an ordering on the edges outgoing from that vertex. For example:
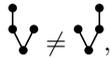
\begin{align*}
	\forestAE \neq \forestBE,
\end{align*}
as the ordering of the edges outgoing from the root are different in the two trees. The vector space spanned by $PT$ is denoted $\mathcal{PT}$, and we write $Lie(PT)$ for the free Lie algebra generated by $\mathcal{PT}$. We now define the \textit{left grafting product} $\graft : \mathcal{PT} \otimes \mathcal{PT} \to \mathcal{PT}$, by letting $t_1 \graft t_2$ be the sum over all trees obtained by adding a leftmost edge from some vertex of $t_2$ to the root of $t_1$. For example:
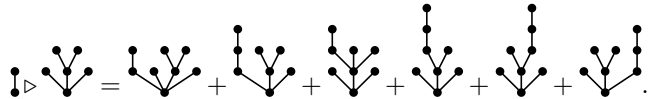
\begin{align*}
	\forestCE \graft \forestDE = \forestEE + \forestFE + \forestGE+\forestHE + \forestIE+\forestJE.
\end{align*}
The left grafting product is now extended to a product $\graft : Lie(PT) \otimes Lie(PT) \to Lie(PT)$ by using the post-Lie axioms:
\begin{align*}
	t_1 \graft [t_2,t_3] = & [t_1 \graft t_2,t_3] + [t_2,t_1 \graft t_3], \\
	[t_1,t_2] \graft t_3 = & t_1 \graft (t_2 \graft t_3) - (t_1 \graft t_2) \graft t_3 - t_2 \graft (t_1 \graft t_3) + (t_2 \graft t_1) \graft t_3.
\end{align*}
The second axiom can be interpreted by identifying $[t_1,t_2]$ with the formal commutator $t_1t_2 - t_2t_1$, which is then grafted onto each vertex of $t_3$. For example:
\begin{align*}
	[\forestKE,\forestLE] \graft \forestME =& \forestNE - \forestOE + \forestPE - \forestQE + \forestRE - \forestSE,
\end{align*}
where we note that $t_1,t_2$ are always grafted together onto the same vertex, with the term $t_1t_2$ having positive sign and the term $t_2t_1$ having negative sign.

\begin{theorem} \cite{MuntheKaasLundervold2012}
	$(Lie(PT),[\cdot,\cdot],\graft)$ is the free post-Lie algebra.
\end{theorem}

\subsection{The Guin-Oudom extension}

Let $(\mathfrak{g},[\cdot,\cdot],\graft)$ be a post-Lie algebra and denote by $U(\mathfrak{g})$ the universal enveloping algebra of the Lie algebra $(\mathfrak{g},[\cdot,\cdot])$. In the special case when the post-Lie algebra is pre-Lie, an extension of the pre-Lie product to $U(\mathfrak{g})$ was constructed by Guin and Oudom in \cite{OudomGuin2008}. The construction was generalized to the post-Lie setting in \cite{EbrahimifardLundervoldMunthekaas2014}, and will be recalled in this section.

For any Lie algebra $\mathfrak{g}$, we have that $U(\mathfrak{g})$ forms a Hopf algebra with concatenation as product and deshuffle as coproduct. The deshuffle coproduct $\Delta_{\shuffle}$ is defined by $\Delta_{\shuffle}(x)=x \otimes 1 + 1 \otimes x$ for $x \in \mathfrak{g}$ an extended to $U(\mathfrak{g})$ as an algebra morphism. For a general element $A \in U(\mathfrak{g})$, we shall use Sweedler's notation $\Delta_{\shuffle}(A) =: A_{(1)} \otimes A_{(2)}$. Now let $A,B,C \in U(\mathfrak{g})$ and let $x,y \in \mathfrak{g}$, then there exists a unique way to extend $\graft$ from $\mathfrak{g}$ to $U(\mathfrak{g})$ given by:
\begin{align*}
	1 \graft A =& A, \\
	xA \graft Y =& x \graft (A \graft y) - (x \graft A) \graft y, \\
	A \graft BC =& (A_{(1)} \graft B)(A_{(2)} \graft C).
\end{align*}
One can now define the \textit{Grossman--Larson} product $\ast$ on $U(\mathfrak{g})$ by
\begin{align*}
	A \ast B := A_{(1)}(A_{(2)} \graft B).
\end{align*}
This product is associative, and can be used to form a second Hopf algebra $(U(\mathfrak{g}),\ast,\Delta_{\shuffle})$ on $U(\mathfrak{g})$. This second Hopf algebra is isomorphic to the universal enveloping algebra of the second Lie algebra $(\mathfrak{g},\llbracket \cdot,\cdot \rrbracket )$.

\subsection{Post-Lie Algebroids} \label{ssect::PostLieAlgebroid}

A \textit{Lie algebroid} is a special case of a Lie-Rinehart algebra, obtained by considering the smooth $\mathbb{R}$-valued functions over a smooth manifold $\mathcal{M}$. The notion of a post-Lie algebroid was considered in \cite{Munthe-KaasLundervold2012}, as a Lie algebroid that is also a post-Lie algebra. Post-Lie algebroids will be an important example of what we later define to be post-Lie-Rinehart algebras. In this section we will recall some basic theory on (post-)Lie algebroids and describe the important example of a post-Lie algebroid on a homogeneous space.

A Lie algebroid is a triple $(E,\llbracket \cdot,\cdot \rrbracket_E, \rho_{\ast})$, where $E$ is a vector bundle over the smooth manifold $\mathcal{M}$, $\llbracket \cdot, \cdot \rrbracket_E$ is a Lie bracket on the module of sections $\Gamma(E)$ and $\rho_{\ast}: E \to T\mathcal{M}$ is a morphism of vector bundles called the \textit{anchor}. We require the bracket and the anchor to satisfy the Leibniz rule
\begin{align*}
	\llbracket x, \phi y \rrbracket_E =& \phi \llbracket x, y \rrbracket_E + (\rho_{\ast} \circ x)(\phi)y,
\end{align*}
where $x,y \in \Gamma(E), \phi \in C^{\infty}(\mathcal{M},\mathbb{R})$ and $(\rho_{\ast} \circ x)(\phi)$ is the Lie derivative of $\phi$ along the vector field $\rho_{\ast} \circ x$. This implies that $\rho_{\ast}$ is a Lie algebra morphism sending $\llbracket \cdot, \cdot \rrbracket_E$ to the Jacobi-Lie bracket of vector fields. It is now straightforward to see that $\Gamma(E)$ is a Lie-Rinehart algebra over the ring $C^{\infty}(\mathcal{M},\mathbb{R})$. A product $\graft: \Gamma(E) \times \Gamma(E) \to \Gamma(E)$ is called a \textit{linear connection} if it satisfies
\begin{align*}
	(\phi x)\graft y =& \phi(x \graft y), \\
	x \graft (\phi y)=& (\rho_{\ast} \circ x)(\phi)y,
\end{align*}
and we say that a Lie bracket $[\cdot,\cdot ]$ is tensorial if it satisfies
\begin{align*}
	[x,\phi y]= [\phi x, y ] = \phi [x,y].
\end{align*}
The authors in \cite{Munthe-KaasLundervold2012} then defines a post-Lie algebroid as a four-tuple $(E,[\cdot,\cdot],\graft,\rho_{\ast})$ such that $E \to \mathcal{M}$ is a vector bundle, $\rho_{\ast}$ is a morphism of vector bundles, $[\cdot,\cdot]$ is tensorial, $\graft$ is a linear connection and $(\Gamma(E),[\cdot,\cdot],\graft)$ is a post-Lie algebra. A post-Lie algebroid is then indeed also a Lie algebroid.
\begin{proposition}\cite{Munthe-KaasLundervold2012}
	Let $(E,[\cdot,\cdot],\graft,\rho_{\ast})$ be a post-Lie algebroid and define $\llbracket x,y \rrbracket_E := x \graft y - y \graft x + [x,y]$. Then $(E,\llbracket\cdot,\cdot \rrbracket_E, \rho_{\ast} )$ is a Lie algebroid.
\end{proposition}

And conversely.

\begin{proposition} \cite{Munthe-KaasLundervold2012}
	Let $(E, \llbracket \cdot,\cdot \rrbracket_E, \rho_{\ast} )$ be Lie algebroid and suppose that $\Gamma(E)$ has linear connection $\graft$ with zero curvature and constant torsion, then $(E, [\cdot,\cdot],\graft \rho_{\ast})$ is a post-Lie algebroid with Lie bracket $[x,y] := -T(x,y) = \llbracket x,y \rrbracket_E - x \graft y + y \graft x$.
\end{proposition}

We will furthermore need the following theorem from \cite{FloystadManchonMuntheKaas2020}.

\begin{theorem} \cite{FloystadManchonMuntheKaas2020}
	Let $(E,\llbracket \cdot,\cdot \rrbracket_E, \rho_{\ast}  )$ be a Lie algebroid over $\mathcal{M}$ and suppose that $\Gamma(E)$ has linear connection $\graft$. Then the Lie-Rinehart algebra $\Gamma(E)$ over the ring $C^{\infty}(\mathcal{M},\mathbb{R})$ is tracial.
\end{theorem}

We conclude this section by showing the concrete construction of a post-Lie algebroid on a homogeneous space. Let $\mathcal{M}$ be a homogeneous space, let $G$ be a Lie group with Lie algebra $\mathfrak{g}$ and a left transitive action $\cdot : G \times \mathcal{M} \to \mathcal{M}$. Then one also gets an action $\cdot : \mathfrak{g} \times \mathcal{M} \to T\mathcal{M}$ by differentiation:
\begin{align*}
	V \cdot p := \frac{d}{d t} \mid_{t=0} \exp(tV)\cdot p,
\end{align*}
where $\exp: \mathfrak{g} \to G$ is the classical exponential map from a Lie algebra to its Lie group. In this setting, any ODE on $\mathcal{M}$ can be formulated in terms of this action. In particular we can write any ODE as:
\begin{align}
	y'(y)=f(y(t))\cdot y(t), \quad y(0)=y_0, \label{eq::LieButcherEquation}
\end{align}
for some map $f: \mathcal{M} \to \mathfrak{g}$. Note that the map $f$ corresponds to a vector field $p \to f(p)\cdot p$. This way to formulate ODEs is one of the main motivations to study post-Lie structures, and underlies a lot of Lie group integration methods \cite{Munthe-KaasWright2008}.

We now endow the space $C^{\infty}(\mathcal{M},\mathfrak{g})$ with a post-Lie algebra structure. For $f \in C^{\infty}(\mathcal{M},\mathfrak{g})$ and $V \in \mathfrak{g}$, we define the \textit{Lie derivative} of $f$ in the direction of $V$ as:
\begin{align*}
	(V \graft f)(p):=\frac{d}{dt} \mid_{t=0} f( \exp(tV)\cdot p).
\end{align*}
Then for $f,g \in C^{\infty}(\mathcal{M},\mathfrak{g})$, we can define the products:
\begin{align}
	[f,g](p):=&[f(p),g(p)]_{\mathfrak{g}}, \label{eq::PointwiseLieBracket} \\
	(f \graft g)(p):=&(f(p) \graft g)(p), \label{eq::VectorFieldGrafting}
\end{align}
where $[\cdot,\cdot]_{\mathfrak{g}}$ is the Lie bracket in $\mathfrak{g}$. This turns $(C^{\infty}(\mathcal{M},\mathfrak{g}),[\cdot,\cdot],\graft)$ into a post-Lie algebra \cite{Munthe-KaasLundervold2012}. Now consider the action algebroid of $\mathfrak{g}$. It is given by the trivial vector bundle $\mathfrak{g} \times \mathcal{M} \to \mathcal{M}$, and anchor map given by the Lie algebra action $\rho_{\ast}(V,p):=(p,V\cdot p) \in T\mathcal{M}$. One easily sees that the sections $\Gamma(\mathfrak{g} \times \mathcal{M})$ can be identified with the space $C^{\infty}(\mathcal{M},\mathfrak{g})$. Then we have that $(\mathfrak{g},[\cdot,\cdot],\graft,\rho_{\ast}  )$ is a post-Lie algebroid, and $C^{\infty}(\mathcal{M},\mathfrak{g})$ is a tracial Lie-Rinehart algebra over $C^{\infty}(\mathcal{M},\mathbb{R})$.

\section{Post-Lie-Rinehart algebras} \label{sec::PostLieRinehart}

Following the definition of a pre-Lie-Rinehart algebra, we naturally arrive at the following definition of a post-Lie-Rinehart algebra.

\begin{definition}
	A post-Lie-Rinehart algebra is a Lie-Rinehart algebra endowed with a flat connection that has constant torsion.
\end{definition}

We immediately note that the Lie-Rinehart algebra $C^{\infty}(\mathcal{M},\mathfrak{g})$ constructed in section \ref{ssect::PostLieAlgebroid} is indeed a tracial post-Lie-Rinehart algebra by this definition. In the present section, we will construct the free tracial post-Lie-Rinehart algebra.

\subsection{The free tracial post-Lie-Rinehart algebra}

We first make some definitions. A \textit{planar aroma} is a connected graph where every vertex has exactly one incoming edge, endowed with an embedding into the plane. This means that, for each vertex, there is a total order on the outgoing edges. The vector space spanned by planar aromas is denoted $\mathcal{PA}$ and the set of all planar aromas is denoted $PA$. An aroma consists of a central cycle, with trees attached to each vertex in this cycle. Note that, because of the planar embedding, edges connecting trees to this cycle can be either "inside" the cycle or "outside" the cycle. A tree inside the cycle will have its root denoted by $\forestTE$ instead of $\forestUE$ to indicate this. This is not a different type of vertex, it is only a notation to indicate the planar embedding. Recall that the edges in the central cycles are oriented right-to-left. As an example, consider the planar aroma:
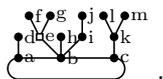
\begin{align*}
	\forestVE.
\end{align*}
The vertex $b$ has outgoing edges to the vertices $e,a,h,i,$ in that order. Because the vertex $e$ is denoted by $\forestWE$, we consider the edge from $b$ to $e$ as being to the left of the edge from $b$ to $a$. The edge from $b$ to $h$ is to the right of the edge from $b$ to $a$.

The symmetric algebra over planar aromas is denoted $S(\mathcal{PA})$. We denote the space $S(\mathcal{PA}) \otimes Lie(\mathcal{PT})$ by $\mathcal{APT}$, \textit{aromatic planar trees}. We will concatenate the planar aromas with the planar trees, rather than writing the tensor product. This space is trivially an $S(\mathcal{PA})$-module, by concatenating more planar aromas. We will write subscript $\mathcal{C}$, e.g. $\mathcal{APT}_{\mathcal{C}}$ or $S(\mathcal{PA}_{\mathcal{C}})$, when we want to emphasize that the vertices of the trees/aromas are labeled by the set $\mathcal{C}$. Such a set always exists, but will most of the time be omitted from the notation. As before, we will label our vertices when required to refer to specific vertices rather than as a way to indicate a generating set $\mathcal{C}$. We now define a tracial post-Lie-Rinehart structure on $\mathcal{APT}$. Let $\alpha_1,\alpha_2 \in S(\mathcal{PA})$ and let $t_1,t_2 \in Lie(\mathcal{PT})$ so that $\alpha_1t_2$ and $\alpha_2t_2$ are two arbitrary elements in $\mathcal{APT}$. Then we define a Lie bracket by:
\begin{align}
	[\alpha_1t_1,\alpha_2t_2]:=\alpha_1\alpha_2[t_1,t_2]. \label{eq::AromaticLieBracket}
\end{align}
We define the anchor map $\rho: \mathcal{APT} \to Der(S(\mathcal{PA}),S(\mathcal{PA}))$. First when $t_1$ is a tree, let $\rho(\alpha_1 t_1)(\alpha_2)$ be the sum over all planar aromas obtained by adding a leftmost edge from some vertex of $\alpha_2$ to the root of $t_2$. For example:
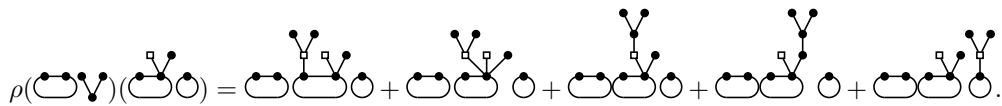
\begin{align*}
	\rho( \forestXE \forestYE )( \forestAF \forestBF )=\forestCF \forestDF \forestEF + \forestFF \forestGF \forestHF + \forestIF \forestJF \forestKF + \forestLF \forestMF \forestNF + \forestOF \forestPF \forestQF.
\end{align*}
Extended to Lie polynomials by:
\begin{align} \label{eq::RhoLiePoly}
	\rho([t_1,t_2])(\alpha)= \rho(t_1)(\rho(t_2)(\alpha)) - \rho(t_1 \graft t_2)(\alpha) - \rho(t_2)(\rho(t_1)(\alpha)) + \rho(t_2 \graft t_1)(\alpha).
\end{align}
Note that this extension is identical to the post-Lie axiom \eqref{eq::PostLie2}. Similar to the free post-Lie algebra, it amounts to grafting the formal commutator $t_1t_2 - t_2t_1$ onto each vertex.
\begin{lemma}
	The equality
	\begin{align*}
		\rho(t)(\alpha_1\alpha_2)=\rho(t)(\alpha_1)\alpha_2+\alpha_1 \rho(t)(\alpha_2)
	\end{align*}
	holds.
\end{lemma}
\begin{proof}
	This is clear from the definition whenever $t_1 \in \mathcal{PT}$. Suppose not, then there exists $t_1,t_2 \in Lie(\mathcal{PT})$ such that $t=[t_1,t_2]$, and:
	\begin{align*}
		\rho([t_1,t_2])(\alpha_1\alpha_2)= \rho(t_1)(\rho(t_2)(\alpha_1\alpha_2)) - \rho(t_1 \graft t_2)(\alpha_1\alpha_2) - \rho(t_2)(\rho(t_1)(\alpha_1\alpha_2)) + \rho(t_2 \graft t_1)(\alpha_1\alpha_2).
	\end{align*}
	The statement then follows by induction over the length of the Lie polynomial, as the induction hypothesis can be applied to each term in the equation above.
\end{proof}
Recall the left grafting product $\graft: Lie(PT) \otimes Lie(PT) \to Lie(PT)$. We now extend it to a product on $\mathcal{APT}$ by grafting onto every vertex:
\begin{align*}
	\alpha_1t_1 \graft \alpha_2t_2=\alpha_1 \rho(t_1)(\alpha_2)t_2 + \alpha_1\alpha_2(t_1 \graft t_2).
\end{align*}

\begin{proposition}
	The triple $(\mathcal{APT},[\cdot,\cdot],\graft)$ is a post-Lie algebra.
\end{proposition}
\begin{proof}
	Let $\alpha_it_i, i=1,2,3$ be three arbitrary elements, we check the two axioms. The first axiom is verified as follows.
	\begin{align*}
		\alpha_1t_1 \graft [\alpha_2t_2,\alpha_3t_3]=& \alpha_1t_1 \graft (\alpha_2\alpha_3[t_2,t_3]) \\
		=&\alpha_1\rho(t_1)(\alpha_2)\alpha_3[t_2,t_3]+\alpha_1\alpha_2\rho(t_1)(\alpha_3)[t_2,t_3] + \alpha_1\alpha_2\alpha_3[t_1 \graft t_2,t_3] + \alpha_1\alpha_2\alpha_3[t_2,t_1\graft t_3] \\
		=& [\alpha_1 \rho(t_1)(\alpha_2)t_2 + \alpha_1 \alpha_2 (t_1 \graft t_2),\alpha_3t_3  ] + [\alpha_2t_2, \alpha_1\rho(t_1)(\alpha_3)t_3 + \alpha_1 \alpha_3 (t_1 \graft t_3)  ] \\
		=& [\alpha_1t_1 \graft \alpha_2t_2,\alpha_3t_3] + [\alpha_2t_2,\alpha_1t_1 \graft \alpha_3t_3].
	\end{align*}
	For the second axiom
	\begin{align}\label{eq::PostLieSecondAxiom}
		[\alpha_1t_1,\alpha_2t_2] \graft \alpha_3t_3=\alpha_1t_1 \graft (\alpha_2t_2 \graft \alpha_3t_3) - (\alpha_1t_1 \graft \alpha_2t_2)\graft \alpha_3t_3 - \alpha_2t_2 \graft (\alpha_1t_1 \graft \alpha_3t_3) + (\alpha_2t_2 \graft \alpha_1t_1)\graft \alpha_3t_3,
	\end{align}
	we divide the computation up into smaller step. Firstly:
	\begin{align*}
		\alpha_1t_1 \graft (\alpha_2t_2 \graft \alpha_3t_3 ) =& \alpha_1t_1 \graft (\alpha_2 \rho(t_2)(\alpha_3)t_3 + \alpha_2\alpha_3(t_2 \graft t_3)  )\\
		=&\alpha_1 \rho(t_1)(\alpha_2)\rho(t_2)(\alpha_3)t_3 + \alpha_1\alpha_2 \rho(t_1)(\rho(t_2)(\alpha_3))t_3 + \alpha_1\alpha_2 \rho(t_2)(\alpha_3)(t_1 \graft t_3)\\
		+& \alpha_1 \rho(t_1)(\alpha_2)\alpha_3(t_2 \graft t_3) + \alpha_1 \alpha_2 \rho(t_1)(\alpha_3)(t_2 \graft t_3) + \alpha_1 \alpha_2 \alpha_3 (t_1 \graft(t_2 \graft t_3)),
	\end{align*}
	and
	\begin{align*}
		(\alpha_1 t_1 \graft \alpha_2 t_2) \graft \alpha_3 t_3 =& \alpha_1\rho(t_1)(\alpha_2)t_2 \graft \alpha_3 t_3 + \alpha_1 \alpha_2 (t_1 \graft t_2) \graft \alpha_3 t_3 \\
		=&\alpha_1 \rho(t_1)(\alpha_2)\rho(t_2)(\alpha_3)t_3 + \alpha_1 \rho(t_1)(\alpha_2)\alpha_3 (t_2 \graft t_3)\\
		+& \alpha_1 \alpha_2 \rho(t_1 \graft t_2)(\alpha_3)t_3 + \alpha_1 \alpha_2 \alpha_3 (t_1 \graft t_2)\graft t_3.
	\end{align*}
	Putting these two computations together, we get:
	\begin{align*}
		ass_{\graft}(\alpha_1t_1,\alpha_2t_2,\alpha_3t_3)=&\alpha_1\alpha_2\rho(t_1)(\rho(t_2)(\alpha_3))t_3+\alpha_1\alpha_2\rho(t_2)(\alpha_3)(t_1 \graft t_3)+\alpha_1\alpha_2\rho(t_1)(\alpha_3)(t_2 \graft t_3)\\
		+&\alpha_1\alpha_2\alpha_3(t_1 \graft (t_2 \graft t_3)) - \alpha_1 \alpha_2 \rho(t_1 \graft t_2)(\alpha_3)t_3 - \alpha_1\alpha_2\alpha_3 (t_1 \graft t_2)\graft t_3.
	\end{align*}
	Swapping $\alpha_1t_1$ and $\alpha_2t_2$, we get after cancellations:
	\begin{align*}
		ass_{\graft}(\alpha_1t_1,\alpha_2t_2,\alpha_3t_3) - ass_{\graft}(\alpha_2t_2,\alpha_1t_1,\alpha_3t_3)=&\alpha_1\alpha_2\rho(t_1)(\rho(t_2)(\alpha_3))t_3 - \alpha_1\alpha_2\rho(t_2)(\rho(t_1)(\alpha_3))t_3\\
		+&\alpha_1\alpha_2\alpha_3(t_1 \graft (t_2 \graft t_3)) - \alpha_1\alpha_2\alpha_3(t_2 \graft (t_1 \graft t_3))\\
		-&\alpha_1\alpha_2\rho(t_1 \graft t_2)(\alpha_3)t_3 + \alpha_1\alpha_2\rho(t_2 \graft t_1)(\alpha_3)t_3 \\
		-&\alpha_1\alpha_2\alpha_3(t_1 \graft t_2)\graft t_3 + \alpha_1 \alpha_2 \alpha_3 (t_2 \graft t_1)\graft t_3.
	\end{align*}
	Finally we compute the left side of the equality \eqref{eq::PostLieSecondAxiom}
	\begin{align*}
		[\alpha_1t_1,\alpha_2t_2] \graft \alpha_3t_3=& (\alpha_1\alpha_2[t_1,t_2])\graft \alpha_3t_3 \\
		=& \alpha_1\alpha_2 \rho([t_1,t_2])(\alpha_3)t_3 + \alpha_1\alpha_2\alpha_3([t_1,t_2]\graft t_3) \\
		=& \alpha_1\alpha_2\rho(t_1)(\rho(t_2)(\alpha_3))t_3 - \alpha_1\alpha_2\rho(t_1 \graft t_2)(\alpha_3)t_3 - \alpha_1\alpha_2\rho(t_2)(\rho(t_1)(\alpha_3))t_3\\
		+& \alpha_1\alpha_2\rho(t_2 \graft t_1)(\alpha_3)t_3 + \alpha_1\alpha_2\alpha_3(t_1 \graft (t_2 \graft t_3)) - \alpha_1\alpha_2\alpha_3((t_1 \graft t_2)\graft t_3)\\
		-& \alpha_1\alpha_2\alpha_3(t_2 \graft (t_1 \graft t_3)) + \alpha_1\alpha_2\alpha_3((t_2 \graft t_1)\graft t_3),
	\end{align*}
	and note that the two sides of the equality agrees.
\end{proof}

As $(\mathcal{APT},[\cdot,\cdot],\graft)$ is a post-Lie algebra, we can define a second Lie bracket $\llbracket \cdot,\cdot \rrbracket$ by:
\begin{align*}
	\llbracket \alpha_1 t_1, \alpha_2 t_2 \rrbracket := \alpha_1 t_1 \graft \alpha_2 t_2 - \alpha_2 t_2 \graft \alpha_1 t_1 + [\alpha_1 t_1, \alpha_2 t_2].
\end{align*}

\begin{lemma}
	The Lie algebra $(\mathcal{APT},\llbracket \cdot, \cdot \rrbracket)$ is a Lie-Rinehart algebra over $S(\mathcal{PA})$ with anchor map $\rho$.
\end{lemma}

\begin{proof}
	We have to check that the map $\rho$ satisfies the Leibniz rule
	\begin{align*}
		\llbracket \alpha_1 t_1, \alpha \alpha_2 t_2 \rrbracket = \rho(\alpha_1 t_1)(\alpha)\alpha_2 t_2 + \alpha \llbracket \alpha_1 t_1, \alpha_2 t_2 \rrbracket,
	\end{align*}
	for $\alpha_1t_1,\alpha_2t_2 \in \mathcal{APT}$ and $\alpha \in S(\mathcal{PA})$. Plugging the definition of the second Lie bracket into the left side of the equality gives:
	\begin{align*}
		\llbracket \alpha_1 t_1, \alpha \alpha_2 t_2 \rrbracket =& \alpha_1 t_1 \graft \alpha \alpha_2 t_2 - \alpha \alpha_2 t_2 \graft \alpha_1 t_1 + \alpha \alpha_1 \alpha_2 [t_1,t_2] \\
		=& \rho(\alpha_1 t_1)(\alpha) \alpha_2 t_2 + \alpha \bigl( \alpha_1 t_1 \graft \alpha_2 t_2 - \alpha_2 t_2 \graft \alpha_1t_1 + [\alpha_1t_1,\alpha_2t_2]   \bigr)\\
		=& \rho(\alpha_1 t_1)(\alpha)\alpha_2 t_2 + \alpha \llbracket \alpha_1 t_1, \alpha_2 t_2 \rrbracket.
	\end{align*}
\end{proof}

\begin{corollary}
	The Lie algebra $(\mathcal{APT},\llbracket \cdot, \cdot \rrbracket)$ is a post-Lie-Rinehart algebra over $S(\mathcal{PA})$ with anchor map $\rho$ and connection $\graft$.
\end{corollary}

We now want to show that $\mathcal{APT}$ is a tracial Lie-Rinehart algebra. First we recall that the connection $\graft$ on $\mathcal{APT}$ extends to a connection on $\hom_{S(\mathcal{PA})}(\mathcal{APT})$ by
\begin{align*}
	(\alpha_1 t_1 \graft \chi)(\alpha_2 t_2)=\alpha_1 t_1 \graft \chi(\alpha_2 t_2) - \chi(\alpha_1 t_1 \graft \alpha_2 t_2 ),
\end{align*}
where $\alpha_1t_1, \alpha_2t_2 \in \mathcal{APT}$ and $\chi \in \hom_{S(\mathcal{PA})}(\mathcal{APT})$. We furthermore recall the map $\delta: \mathcal{APT} \to \hom_{S(\mathcal{PA})}(\mathcal{APT})$ given by
\begin{align*}
	\delta(\alpha_1 t_1)(\alpha_2 t_2) = \alpha_2 t_2 \graft \alpha_1 t_1,
\end{align*}
and the algebra of elementary $S(\mathcal{PA})$-module endomorphisms $E\ell_{S(\mathcal{PA})}(\mathcal{APT})$, the subalgebra of $\hom_{S(\mathcal{PA})}(\mathcal{APT})$ generated by $\{ \alpha_1t_1 \graft \dots \graft \alpha_nt_n \graft \delta (\alpha t): \alpha t, \alpha_1t_1 \dots \alpha_n t_n \in \mathcal{APT}  \}$. In the pre-Lie case \cite[Proposition 2.10]{FloystadManchonMuntheKaas2020}, the generating set of $E\ell_R(L)$ can be greatly reduced to $\{\delta X: X \in L \}$. Such a result does not hold in the post-Lie case. Denote by
\begin{align*}
	\tilde{\delta}(\alpha_1t_1)(\alpha_2 t_2):=& \alpha_2 t_2 \bgraft \alpha_1t_1\\
	=& \alpha_2 t_2 \graft \alpha_1 t_1 + [\alpha_2t_2,\alpha_1t_1]
\end{align*}
the $\delta$-map induced by the sub-adjacent post-Lie product. 

\begin{lemma} \label{Lemma::ElementarySubAlg}
	Let $L$ be a post-Lie-Rinehart algebra over $R$, with anchor map $\rho$ and connection $\graft$. Then $E\ell_R(L)$ is a subalgebra of the algebra generated by  $\{\delta X,\tilde{\delta}X: X \in L  \}$.
\end{lemma}

\begin{proof}
	First we show that elements of the form $Y \graft \delta X$ can be generated by elements of the form $\delta Z,\tilde{\delta}Z$, $Z \in L$:
	\begin{align*}
		(Y \graft \delta X)(Z)=&Y \graft (Z \graft X) - (Y \graft Z) \graft X \\
		=&[Y,Z] \graft X + Z \graft (Y \graft X) - (Z \graft Y) \graft X \\
		=&(\delta(Y \graft X) - \delta X \circ \tilde{\delta}Y )(Z).
	\end{align*}
	We next show that elements of the form $Y \graft \tilde{\delta}X$ can be generated by elements of the form $\delta Z,\tilde{\delta}Z$, $Z \in L$:
	\begin{align*}
		(Y \graft \tilde{\delta}X)(Z)=&Y \graft (Z \graft X + [Z,X])-(Y\graft Z) \graft X - [Y \graft Z,X]\\
		=&[Y,Z]\graft X + Z \graft (Y \graft X) - (Z \graft Y)\graft X + [Z,Y \graft X]\\
		=&(\tilde{\delta}(Y \graft X) - \delta X \circ \tilde{\delta}Y )(Z).
	\end{align*}
	For elements of the form $Y_1\graft Y_2 \graft \delta X$, the statement now follows by Leibniz rule and the above computations:
	\begin{align*}
		Y_1 \graft Y_2 \graft \delta X = Y_1 \graft \delta (Y_2 \graft X) - (Y_1 \graft \delta X) \circ \tilde{\delta}Y_2 - \delta X \circ (Y_1 \graft \tilde{\delta}Y_2),
	\end{align*}
	and we may continue in this way for arbitrary elements in $E\ell_R(L)$.
\end{proof}

Note that the inclusion in the above Lemma is strict, as already $\tilde{\delta}X \notin E\ell_{R}(L)$ for any $X \in L$. In fact, we show that $E\ell_R(L)$ is exactly the subalgebra consisting of elements that has $\delta X$ as leftmost term.

\begin{proposition}
	Let $L$ be a post-Lie-Rinehart algebra over $R$, with anchor map $\rho$ and connection $\graft$, then $E\ell_R(L) =\{\delta X: X \in L \} \circ gen(\{\delta X,\tilde{\delta}X:X \in L  \})$, where $gen(A)$ means the algebra generate by the set $A$ using composition as product. These are all the ways to repeatedly compose elements of the form $\delta X,\tilde{\delta}X$ such that the leftmost term is of the form $\delta X$.
\end{proposition}
\begin{proof}
	The inclusion $E\ell_R(L) \subseteq \{\delta X: X \in L \} \circ gen(\{\delta X,\tilde{\delta}X:X \in L  \})$ is clear from the proof of Lemma \ref{Lemma::ElementarySubAlg}. \\
	The inclusion $\{\delta X: X \in L \} \circ gen(\{\delta X,\tilde{\delta}X:X \in L  \}) \subseteq E\ell_R(L)$ follows if we can show that $\delta X \circ \tilde{\delta}Y_1 \circ \dots \circ \tilde{\delta}Y_k \in E\ell_R(L)$, for all $X,Y_1,\dots,Y_k \in L$. The cases of $k=1,2$ can be deduced from the computations in the proof of Lemma \ref{Lemma::ElementarySubAlg}. Suppose that the statement holds for $k=1,\dots,n-1$, and consider
	\begin{align*}
		\delta X \circ \tilde{\delta}Y_1 \circ \dots \circ \tilde{\delta}Y_n=&d(Y_1 \graft X)\circ \tilde{\delta}Y_2 \circ \dots \circ \tilde{\delta}Y_n - (Y_1 \graft \delta X)\circ \tilde{\delta}Y_2 \circ \dots \circ \tilde{\delta}Y_n.
	\end{align*}
	The first term is clearly in $E\ell_R(L)$ by the induction hypothesis, the second term satisfies:
	\begin{align*}
		(Y_1 \graft \delta X)\circ \tilde{\delta}Y_2 \circ \dots \circ \tilde{\delta}Y_n=& Y_1 \graft \Bigl( \delta X \circ \tilde{\delta}Y_2 \circ \dots \circ \tilde{\delta} Y_n \Bigr) - \delta X \circ (Y_1 \graft \tilde{\delta}Y_2)\circ \tilde{\delta}Y_3 \circ \dots \circ \tilde{\delta}Y_n\\ 
		-& \dots - \delta X \circ \tilde{\delta} Y_2 \circ \dots \circ \tilde{\delta} Y_{n-1} \circ (Y_1 \graft \tilde{\delta} Y_n).
	\end{align*}
	Again the first term here is clearly in $E\ell_R(L)$. The remaining terms follow by noting the identity
	\begin{align*}
		X \graft \tilde{\delta} Y = \tilde{\delta}(X \graft Y) - \delta Y \circ \tilde{\delta} X.
	\end{align*}
	The proposition follows by mathematical induction.
\end{proof}

We will in future proofs require different ways of representing the algebra of elementary $S(\mathcal{PA})$-endomorphisms. One of these way will be via the Guin--Oudom extension as follows.

\begin{lemma} \label{Lemma::EellAsGuinOudom}
	Let $L$ be a post-Lie-Rinehart algebra over $R$, with anchor map $\rho$ and connection $\graft$. Let $Z \in L$ and let $Y_1 \graft \dots \graft Y_n \graft \delta X \in E\ell_R(L)$ be arbitrary elements. Then:
	\begin{align} \label{eq::EellAsGuinOudom}
		(Y_1 \graft \dots \graft Y_n \graft \delta X)(Z)= (Y_1 \ast \dots \ast Y_n) Z \graft X,
	\end{align}
	where $(Y_1 \ast \dots \ast Y_n) Z \in U(L)$, $\graft: U(L) \otimes L \to L$ is the Guin--Oudom extension of $\graft: L \otimes L \to L$, and $\ast$ is the Grossman--Larson product.
\end{lemma}

\begin{proof}
	Recall that the Guin--Oudom extension of $\graft$ is given by
	\begin{align*}
		x_1\dots x_k \graft y = x_1 \graft (x_2 \dots x_k \graft y) - \sum_{i=2}^k x_2 \dots x_{i-1}(x_1 \graft x_i) x_{i+1} \dots x_k \graft y.
	\end{align*}
	First we show \eqref{eq::EellAsGuinOudom} for $n=1,2$:
	\begin{align*}
		(Y_1 \graft \delta X)(Z)=&Y_1 \graft (Z\graft X) - (Y_1 \graft Z) \graft X \\
		=& Y_1Z \graft X, \\
		(Y_1 \graft Y_2 \graft \delta X)(Z) =& Y_1 \graft (Y_2 \graft \delta X )(Z) - (Y_2 \graft \delta X)(Y_1 \graft Z)\\
		=&Y_1 \graft (Y_2 \graft (Z \graft X) - (Y_2 \graft Z)\graft X ) - Y_2 \graft ((Y_1 \graft Z) \graft X) + (Y_2 \graft (Y_1 \graft Z)) \graft X \\
		=&Y_1 \graft (Y_2Z \graft X ) - Y_2(Y_1 \graft Z) \graft X \\
		=&Y_1Y_2Z \graft X + (Y_1 \graft Y_2)Z \graft X\\
		=&(Y_1 \ast Y_2)Z \graft X.
	\end{align*}
	Now suppose that \eqref{eq::EellAsGuinOudom} holds for all $n=1,\dots,k-1$, then:
	\begin{align*}
		(Y_1 \graft \dots \graft Y_k \graft \delta X)(Z)=&Y_1 \graft (Y_2 \graft \dots \graft Y_k \graft \delta X)(Z) - (Y_2 \graft \dots \graft Y_k \graft \delta X)(Y_1 \graft Z) \\
		=&Y_1 \graft( (Y_2 \ast \dots \ast Y_k)Z \graft X ) - (Y_2 \ast \dots \ast Y_k)(Y_1 \graft Z) \graft X \\
		=&Y_1(Y_2 \ast \dots \ast Y_k)Z \graft X + (Y_1 \graft (Y_2 \ast \dots \ast Y_k))Z \graft X \\
		=&(Y_1 \ast \dots \ast Y_k)Z \graft X.
	\end{align*}
	The statement follows by mathematical induction.
\end{proof}

We will further show that $E\ell_{S(\mathcal{PA})}(\mathcal{APT})$ can be represented by \textit{marked aromatic planar trees}. A marked aromatic planar tree is a triple $(v,n,t)$ where $t\in \mathcal{APT}$ is an aromatic planar tree, $v$ is a vertex in $y$ and $n$ is an index in $\{1,\dots, |v|+1\}$ where $|v|$ is the number of edges outgoing from $v$. The vector space spanned by marked aromatic planar trees is denoted $\mathcal{MAPT}$. The subspace spanned by $(v,n,t),$ $t \in \mathcal{PT}$ are called marked planar trees and denoted $\mathcal{MPT}$. We similarly define marked planar aromas, $\mathcal{MPA}$, as triples $(v,n,\alpha)$ where $\alpha \in \mathcal{PA}$. There is an injective map $\mathcal{MAPT} \to \hom(\mathcal{APT})$ given by
\begin{align*}
	(v,n,t) \to (\alpha_1 t_1 \to \alpha_1 t_1 \graft_{v,n} t ),
\end{align*}
where $\alpha_1 t_1 \graft_{v,n} t$ is grafting the root of $t_1$ onto the vertex $v$ in $t$ in the $n$:th position with respect to the planar embedding of the edges outgoing from $v$. As an example:
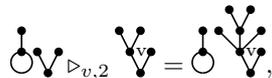
\begin{align*}
	\forestRF\forestSF \graft_{v,2} \forestTF=\forestUF\forestVF,
\end{align*}
where we emphasize that the grafting onto the vertex $v$ is not left grafting. Indeed it is evident by Lemma \ref{Lemma::EellAsGuinOudom} that $E\ell_{S(\mathcal{PA})}(\mathcal{APT})$ will contain maps that grafts in all planar positions, as the $Y_1,\dots,Y_n$ are to the left of $Z$ in the grafting in \eqref{eq::EellAsGuinOudom}. Similarly there is a map $\mathcal{MPA} \to \hom(\mathcal{PT},\mathcal{PA})$ given by
\begin{align*}
	(v,n,\alpha) \to (t \to t \graft_{v,n} \alpha).
\end{align*}
Composition of maps then induces the product:
\begin{align*}
	(v_1,n_1,t_1) \circ (v_2,n_2,t_2)=(v_2,n_2,t_2 \graft_{v_1,n_1} t_1)
\end{align*}
on $\mathcal{MPT}$. Now consider the linear map $\tau: \mathcal{MPT} \to \mathcal{PA}$, which maps a marked planar tree $(v,n,t)$ to a planar aroma by grafting the root of $t$ onto the vertex $v$ in the $n$:th position with respect to the planar embedding. This also naturally generalizes to a map $\tau: \mathcal{MAPT} \to S(\mathcal{PA})$. For example:
\begin{align*}
	\tau(v,3,\forestWF\forestXF)=& \forestYF \forestAG, \\
	\tau(v,2,\forestBG \forestCG)=& \forestDG.
\end{align*}

We can now generalize the representation of aromas as marked trees \cite[Lemma 3.3]{FloystadManchonMuntheKaas2020} to the post-Lie setting.

\begin{lemma} \label{Lemma::PAisMPTQuotient}
	The vector space $\mathcal{PA}$ can be naturally identified with the quotient $\mathcal{MPT}/[\mathcal{MPT},\mathcal{MPT}]$, where $[\mathcal{MPT},\mathcal{MPT}]$ is the vector space spanned by the commutators in $\mathcal{MPT}$, so that the map $\tau$ is the natural projection from $\mathcal{MPT}$ onto $\mathcal{MPT}/[\mathcal{MPT},\mathcal{MPT}]$.
\end{lemma}

\begin{proof}
	The proof of \cite[Lemma 3.3]{FloystadManchonMuntheKaas2020} can be copied essentially word for word. 
\end{proof}

Now consider the subspace $\Upsilon \subseteq \hom(\mathcal{PT},\mathcal{PA})$ of maps $\Upsilon =  \{ t \to (\omega t \graft \alpha ): \omega \in \mathcal{OF},\alpha \in \mathcal{PA}   \}$. One then easily sees that $\Upsilon$ is a subspace of $\mathcal{MPA}$. Let $\Psi= \{\alpha t : \alpha \in \Upsilon, t \in \mathcal{PT}  \}$, which is a subspace of $\mathcal{MAPT}$ where the mark is on the aromas. We now have compositions:
\begin{align}
	\mathcal{MPT} \circ \mathcal{MPT} \to \mathcal{MPT}:& (v_1,n_1,t_1) \circ (v_2,n_2,t_2)=(v_2,n_2,t_2 \graft_{v_1,n_1} t_1), \label{eq::MPATComp1} \\
	\Psi \circ \mathcal{MPT}:& (v_1,n_1,\alpha)t_1 \circ (v_2,n_2,t_2)=(v_2,n_2,t_2 \graft_{v_1,n_1} \alpha)t_1, \label{eq::MPATComp2} \\
	\mathcal{MPT} \circ \Psi \to \Psi:& (v_1,n_1,t_1) \circ (v_2,n_2,\alpha_2)t_2 = (v_2,n_2,\alpha_2) t_2 \graft_{v_1,n_1} t_1, \label{eq::MPATComp3}\\
	\Psi \circ \Psi \to S(\mathcal{PA}) \otimes \Psi:& (v_1,n_1,\alpha_1)t_1 \circ (v_2,n_2,\alpha_2)t_2 = t_2 \graft_{v_1,n_1} \alpha_1 \otimes (v_2,n_2,\alpha_2)t_1. \label{eq::MPATComp4}
\end{align}

\begin{lemma}
	The $S(\mathcal{PA})$-submodule
	\begin{align*}
		\Omega = (S(\mathcal{PA}) \otimes \mathcal{MPT} ) \bigoplus ( S(\mathcal{PA}) \otimes \Psi ) \bigoplus (S(\mathcal{PA}) \otimes (\Psi \circ \mathcal{MPT}) )
	\end{align*}
	of $\mathcal{MAPT}$ is a subalgebra and the decomposition is free.
\end{lemma}
\begin{proof}
	It is clear by the compositions $\eqref{eq::MPATComp1}-\eqref{eq::MPATComp4}$ that $\Omega$ is closed under composition, and hence is a subalgebra. To show the freeness of the decomposition, note that an element in $S(\mathcal{PA}) \otimes \mathcal{MPT}$ has its mark on a tree and hence cannot be a linear combination of the other two terms. An element of $\Psi$ has at least one term with a mark on the central cycle of an aroma, and hence also cannot be a linear combination of the other two terms.
\end{proof}

This now lets us represent $E\ell_{S(\mathcal{PA})}(\mathcal{APT})$ as marked aromatic planar trees.

\begin{proposition} \label{prop::OmegaIsEell}
	The algebra $E\ell_{S(\mathcal{PA})}(\mathcal{APT})$ coincides with the algebra $\Omega$.
\end{proposition}
\begin{proof}
	The inclusion $E\ell_{S(\mathcal{PA})}(\mathcal{APT}) \subseteq \Omega$ is clear by Lemma \ref{Lemma::EellAsGuinOudom}. Let $\alpha' t' \to ( (\alpha_1 t_1 \ast \dots \ast \alpha_n t_n)\alpha' t' \graft \alpha t  )$ be an arbitrary element of $E\ell_{S(\mathcal{PA})}(\mathcal{PAT})$, and note:
	\begin{align}
		(\alpha_1 t_1 \ast \dots \ast \alpha_n t_n)\alpha' t' \graft \alpha t = ((\alpha_1 t_1 \ast \dots \ast \alpha_n t_n)\alpha' t' \graft \alpha) t + \alpha( (\alpha_1 t_1 \ast \dots \ast \alpha_n t_n)\alpha' t' \graft t), \label{eq::ElInPsi}
	\end{align}
	The first term is the evaluation of a map in $\Psi$ and the second term is the evaluation of a map in $\mathcal{MPT}$.\\
	The inclusion $\mathcal{MPT} \subseteq E\ell_{S(\mathcal{PA})}(\mathcal{APT})$: If tree $t$ consists only of the vertex $v$, then $(v,1,t)=\delta(t) \in E\ell_{S(\mathcal{PA})}(\mathcal{APT})$. Suppose that $(v_k,n_k,t_k) \in E\ell_{S(\mathcal{PA})}(\mathcal{APT})$ for all $t_k$ with at most $k$ vertices and let $(v,n,t)$ have $k+1$ vertices. If $v$ is not the root of $t$, then $(v,n,t)=(v',n',t')\circ (v,n,t_v) \in E\ell_{S(\mathcal{PA})}(\mathcal{APT})$, where $t_v$ is the subtree of $t$ that has $v$ as a root and $t'$ is the subtree obtained by removing $t_v$ from $t$. If $v$ is the root of $t$, then $(v,1,t)=\delta(v) - \sum_{w \neq v}(w,1,t) \in E\ell_{S(\mathcal{PA})}(\mathcal{APT})$. Lastly suppose that $v$ is the root of $t$ and that $(v,n,t) \in E\ell_{S(\mathcal{PA})}(\mathcal{APT})$ for all $n=1,\dots,q-1$. Let $t_1$ be the leftmost branch attached to $v$ and let $t'$ be the tree obtained by removing $t_1$ from $t$, then $(v,q,t)=\sum_{w \neq v} (v,q-1,t_1 \graft_w t') - t_1 \graft (v,q,t) \in E\ell_{S(\mathcal{PA})}(\mathcal{APT})$. The inclusion $\mathcal{MPT} \subseteq E\ell_{S(\mathcal{PA})}(\mathcal{APT})$ now follows by mathematical induction. \\
	The inclusion $\Psi \subseteq E\ell_{S(\mathcal{PA})}(\mathcal{PAT})$:
	\begin{align*}
		(\omega t' \graft \alpha)t = (\omega t') \graft \alpha t - \alpha (\omega t' \graft t).
	\end{align*}
\end{proof}

\begin{corollary}
	The algebra $E\ell_{S(\mathcal{PA})}(\mathcal{APT})$ is a subalgebra of $\mathcal{MAPT}$.
\end{corollary}

As $E\ell_{S(\mathcal{PA})}(\mathcal{APT})$ is a subalgebra of $\mathcal{MAPT}$, this in particular means that the map $\tau$ is defined on it. This is indeed what we need to show that the post-Lie-Rinehart algebra $\mathcal{APT}$ is tracial.

\begin{lemma}
	\begin{enumerate}
		\item $\Omega$ is an $\mathcal{APT}$-submodule of $End_{S(\mathcal{PA})}(\mathcal{APT})$.
		\item The map $\tau: \mathcal{MAPT} \to S(\mathcal{PA})$ is an $\mathcal{APT}$-module map.
		\item $\tau$ is a trace map.
	\end{enumerate}
\end{lemma}
\begin{proof}
	\begin{enumerate}
		\item For an aromatic tree $\alpha_1 t_1$ and a marked aromatic tree $(v,n,t)$, one obtains
		\begin{align*}
			\alpha_1 t_1 \graft (v,n,t):= \sum_{w} (v,n+\delta_{v=w},\alpha_1t_1 \graft_w t ).
		\end{align*}
		\item The proof of \cite[Lemma 3.11]{FloystadManchonMuntheKaas2020} applies.
		\item The proof of \cite[Lemma 3.11]{FloystadManchonMuntheKaas2020} applies.
	\end{enumerate}
\end{proof}

\begin{corollary}
	$\mathcal{APT}$ is a tracial post-Lie-Rinehart algebra.
\end{corollary}

Having established that $\mathcal{APT}$ is a tracial post-Lie-Rinehart algebra, we conclude by showing that it is free.

\begin{theorem}
	Let $(L,R)$ be a tracial post-Lie-Rinehart algebra and $\mathcal{C} \to L$ a map of sets. Then this map extends to a unique homomorphism of tracial post-Lie-Rinehart algebras
	\begin{align*}
		(\zeta,\gamma): (\mathcal{APT}_{\mathcal{C}},S(\mathcal{PA}_{\mathcal{C}})) \to (L,R).
	\end{align*}
	Furthermore there exists a morphism $\beta$ of associative algebras such that the following diagram commutes:
	\begin{center}
		\begin{tikzcd} 
			\mathcal{APT}_{\mathcal{C}} \arrow[r, "\delta"] \arrow[d, "\zeta"] & E\ell_{S(\mathcal{PA}_{\mathcal{C}})}(\mathcal{APT}_{\mathcal{C}}) \arrow[r,"\tau"] \arrow[d, "\beta"] & S(\mathcal{PA}_{\mathcal{C}}) \arrow[d, "\gamma"] \\
			L \arrow[r, "\delta"] &  E\ell_R(L) \arrow[r,"\tau"] & R.
	\end{tikzcd}\end{center}
	It fulfills the following properties for $\phi \in E\ell_{S(\mathcal{PA}_{\mathcal{C}})}(\mathcal{APT}_{\mathcal{C}})$ and $\alpha t \in \mathcal{APT}_{\mathcal{C}}$:
	\begin{align}
		\beta(\phi)(\zeta(\alpha t)  )=& \zeta(\phi(\alpha t) ) \label{eq::Freeness1}\\
		\beta(\alpha t \graft \phi )=& \zeta(\alpha t) \graft \beta(\phi), \label{eq::Freeness2} \\
		\gamma (\alpha t \graft \alpha)=&\zeta(\alpha t) \graft \gamma(\alpha). \label{eq::Freeness3}
	\end{align}
\end{theorem}

\begin{proof}
	We begin by constructing the maps $\zeta,\beta,\gamma$. First consider the subspace $Lie(\mathcal{PT}_{\mathcal{C}})$ of $\mathcal{APT}_{\mathcal{C}}$, we define the restriction $\zeta|_{Lie(\mathcal{PT}_{\mathcal{C}})}$ to be the unique post-Lie algebra morphism induced by $Lie(\mathcal{PT}_{\mathcal{C}})$ being the free post-Lie algebra. Next we show that there is a unique way to define the restriction $\beta|_{\mathcal{MPT}_{\mathcal{C}} }$. If the tree $t$ consists only of the vertex $v$, then it is clear we must define
	\begin{align*}
		\beta(v,1,t)(x):=d(\zeta(t))(x)=x \graft \zeta(t)
	\end{align*}
	for all $x \in L$. Now suppose $\beta$ has been defined for all trees up to $k-1$ vertices, and let $(v,n,t)\in \mathcal{MPT}_{\mathcal{C}}$ have $k$ vertices. If $v$ is not the root of $t$ then $(v,n,t)$ can be decomposed by $(v,n,t)=(v',n',t')\circ (v,n,t_v)$ and $\beta$ has to be defined by
	\begin{align} \label{eq::BetaComposition}
		\beta(v,n,t):=\beta( v',n',t' ) \circ \beta(v,n,t_v),
	\end{align}
	where $t_v$ is the subtree of $t$ that has $v$ as root, $t'$ is the subtree obtained by removing $t_v$ from $t$ and $v',n'$ is the vertex and position that $t_v$ is grafted on in $t$. If $t$ has $k$ vertices and $v$ is the root of $t$, then we must define
	\begin{align*}
		\beta(v,1,t):=\delta(\zeta(t)) - \sum_{w \neq v} \beta(w,1,t).
	\end{align*}
	Finally suppose that $v$ is the root of $t$ and we have defined $\beta(v,n,t)$ for all $n=1,\dots,q-1$. Further let $t_1$ be the leftmost branch attached to $v$, and let $t'$ be the tree obtained by removing $t_1$ from $t$. Then:
	\begin{align*}
		t_1 \graft (v,q-1,t')=&(v,q,t) + \sum_{w \neq v} (v,q-1,t_1 \graft_w t'),
	\end{align*}
	hence we must define:
	\begin{align*}
		\beta(v,q,t):= \zeta(t_1) \graft \beta(v,q-1,t') - \sum_{w \neq v} \beta(v,q-1,t_1 \graft_w t').
	\end{align*}
	This has now uniquely defined $\beta$ on all of $\mathcal{MPT}_{\mathcal{C}}$. Next we show that $\beta|_{\mathcal{MPT}_{\mathcal{C}}}$ is an algebra morphism. Let $(v_1,n_1,t_1),(v_2,n_2,t_2)$ be arbitrary. If $v_2$ is the root of $t_2$, then the morphism property $\beta( (v_1,n_1,t_1) \circ (v_2,n_2,t_2) )= \beta( v_1,n_1,t_1 ) \circ \beta(v_2,n_2,t_2 )$ is by construction \eqref{eq::BetaComposition}. Suppose that $v_2$ is not the root of $t_2$, then $(v_2,n_2,t_2)=(v_3,n_3,t_3) \circ (v_2,n_2,t_2^{v_2})$, where $t_2^{v_2}$ is the subtree of $t_2$ that has $v_2$ as root. Either $v_3$ is the root of $t_3$, or we can continue in this way to decompose $(v_2,n_2,t_2)=(v_m,n_m,t_m^{v_m}) \circ \dots \circ (v_2,n_2,t_2^{v_2})$, where $v_j$ is the root of $t_j^{v_j}$ for all $j=2,\dots,m$. Then by associativity:
	\begin{align*}
		\beta( (v_1,n_1,t_1) \circ (v_2,n_2,t_2) )=& \beta(  (v_1,n_1,t_1) \circ (v_m,n_m,t_m^{v_m}) \circ \dots \circ (v_2,n_2,t_2^{v_2})  ) \\
		=& \beta( (v_1,n_1,t_1) \circ (v_m,n_m,t_m^{v_m}) \circ \dots \circ (v_3,n_3,t_2^{v_3}) ) \circ \beta(v_2,n_2,t_2^{v_2}) \\
		=& \beta( v_1,n_1,t_1 ) \circ \beta(v_m,n_m,t_m^{v_m} ) \circ \dots \circ \beta(v_2,n_2,t_2^{v_2}) \\
		=&\beta(v_1,n_1,t_1) \circ \beta((v_m,n_m,t_m^{v_m}) \circ (v_{m-1},n_{m-1},t_{m-1}^{v_{m1}}) ) \circ \dots \circ \beta(v_2,n_2,t_2^{v_2}) \\
		=&\beta( v_1,n_1,t_1 ) \circ \beta(v_2,n_2,t_2 ).
	\end{align*}
	Showing that \eqref{eq::Freeness2} holds on this subspace is done in exactly the same way. If $v_2$ is the root of $t_2$, then $\beta(t_1 \graft (v_2,n_2,t_2) )= \zeta(t_1) \graft \beta(v_2,n_2,t_2)$ is by construction. If $v_2$ is not the root of $t_2$, then we can decompose $(v_2,n_2,t_2)=(v_m,n_m,t_m^{v_m}) \circ \dots \circ (v_2,n_2,t_2^{v_2})$ as above, and property \eqref{eq::Freeness2} follows immediately from the Leibniz rule. Now since $\beta$ is an algebra morphism, it induces a map
	\begin{align*}
		\overline{\beta}: \mathcal{MPT}_{\mathcal{C}} / [\mathcal{MPT}_{\mathcal{C}}, \mathcal{MPT}_{\mathcal{C}}  ] \to \hom_R(L) / [\hom_R(L),\hom_R(L)  ].
	\end{align*}
	Then, since $\tau$ is a trace, it induces a map
	\begin{align*}
		\overline{\tau}: \hom_R(L)/[\hom_R(L),\hom_R(L)] \to R.
	\end{align*}
	Then we can define the restriction $\gamma|_{\mathcal{PA}_{\mathcal{C}}}$ by $\gamma|_{\mathcal{PA}_{\mathcal{C}}}:= \overline{\tau} \circ \overline{\beta}$. Now since $S(\mathcal{PA}_{\mathcal{C}})$ is a free commutative algebra, there is the unique way to extend $\gamma$ to the full space:
	\begin{align*}
		\gamma(\alpha_1 \dots \alpha_n):= \gamma(\alpha_1) \dots \gamma (\alpha_n).
	\end{align*}
	We now get to extend $\zeta$ to all of $\mathcal{APT}_{\mathcal{C}}$ by:
	\begin{align*}
		\zeta(\alpha t) = \gamma(\alpha)\zeta(t).
	\end{align*}
	It is straightforward to check that $\zeta$ is a post-Lie morphism also on $\mathcal{APT}_{\mathcal{C}}$. We can now define $\beta$ on $\Psi$. First note that an arbitrary element in $\Psi$ can be written as a map \begin{align*}
	t' \to ((\alpha_1t_1 \ast \dots \ast \alpha_n t_n)t' \graft \alpha )t,
	\end{align*} and recall identity \eqref{eq::ElInPsi}. Hence we define:
	 \begin{align}
		\beta( t' \to ((\alpha_1t_1 \ast \dots \ast \alpha_n t_n)t' \graft \alpha )t ):= \zeta(\alpha_1t_1) \graft \dots \graft \zeta(\alpha_n t_n) \graft \beta(\delta (\alpha t)) - \gamma(\alpha)\bigl(\zeta(\alpha_1t_1) \graft \dots \graft \zeta(\alpha_n t_n) \graft \beta(\delta t)  \bigr). \label{eq::BetaAroma}
	\end{align}
	By using Lemma \ref{Lemma::EellAsGuinOudom}, we see that this definition coincides with the map
	\begin{align*}
		\beta( t' \to ((\alpha_1t_1 \ast \dots \ast \alpha_n t_n)t' \graft \alpha )t ) = Z \to \bigl( ( \zeta(\alpha_1t_1) \ast \dots \ast \zeta(\alpha_nt_n))Z \graft \gamma(\alpha)   \bigr) \zeta(t),
	\end{align*}
	for $Z \in L$. We can now extend $\beta$ to $\Psi \circ \mathcal{MPT}$ as an algebra morphism. This defines $\beta$ on all of $E\ell_{S(\mathcal{PA}_{\mathcal{C}})}(\mathcal{APT}_{\mathcal{C}})$. It remains to check that $\beta$ is an algebra morphism for compositions $\Psi \circ \Psi$ and compositions $\mathcal{MPT} \circ \Psi$. Before we proceed with this, we first note that $\beta$ satisfies the equality:
	\begin{align*}
		\beta(Y_1 \graft \dots \graft Y_n \graft \delta X )=\zeta(Y_1) \graft \dots \graft \zeta(Y_n) \graft \delta \zeta(X),
	\end{align*}
	which is by construction for $X$ a tree and then by equation \eqref{eq::BetaAroma} for $X$ an aromatic tree. Property \eqref{eq::Freeness2} will then follow once we show the morphism property. For compositions $\Psi \circ \Psi$, we have by \eqref{eq::MPATComp4}:
	\begin{align*}
		&\beta\Bigl( \bigl( \alpha_1t_1 \graft \dots \graft \alpha_nt_n \graft \delta(\alpha t) - \alpha (\alpha_1 t_1 \graft \dots \graft \alpha_n t_n \graft \delta (t) )  \bigr) \circ \bigl( \alpha'_1 t'_1 \graft \dots \alpha'_m t'_m \graft \delta(\alpha' t') - \alpha' ( \alpha'_1 t'_1 \graft \dots \graft \alpha'_m t'_m \graft \delta(t')  )   \bigr) \Bigr)\\
		=&\beta\Bigl( \bigl( (\alpha_1t_1 \ast \dots \ast \alpha_n t_n)t' \graft \alpha  \bigr) \cdot \bigl( \alpha'_1 t'_1 \graft \dots \alpha'_m t'_m \graft \delta(\alpha' t) - \alpha' ( \alpha'_1 t'_1 \graft \dots \graft \alpha'_m t'_m \graft \delta(t)  )  \bigr)   \Bigr) \\
		=&\gamma\bigl( (\alpha_1t_1 \ast \dots \ast \alpha_n t_n)t' \graft \alpha \bigr) \cdot \Bigl( \zeta(\alpha'_1t'_1)\graft \dots \graft \zeta(\alpha'_m t'_m)\graft \delta \zeta(\alpha' t) - \gamma(\alpha')\bigl( \zeta(\alpha'_1t'_1) \graft \dots \graft \zeta(\alpha'_mt'_m) \graft \delta \zeta(t) \bigr)  \Bigr) \\
		=& \gamma\bigl( (\alpha_1t_1 \ast \dots \ast \alpha_n t_n)t' \graft \alpha \bigr) \cdot \Bigl( Z \to ( (\zeta(\alpha'_1t'_1) \ast \dots \ast \zeta(\alpha'_mt'_m) )Z \graft \gamma(\alpha') )  \Bigr) \zeta(t) \\
		=& \beta( \alpha_1t_1 \graft \dots \graft \alpha_nt_n \graft \delta(\alpha t) - \alpha (\alpha_1 t_1 \graft \dots \graft \alpha_n t_n \graft \delta (t) )) \circ \Bigl( Z \to ( (\zeta(\alpha'_1t'_1) \ast \dots \ast \zeta(\alpha'_mt'_m) )Z \graft \gamma(\alpha')  \Bigr) \zeta(t')\\
		=&\beta( \alpha_1t_1 \graft \dots \graft \alpha_nt_n \graft \delta(\alpha t) - \alpha (\alpha_1 t_1 \graft \dots \graft \alpha_n t_n \graft \delta (t) )) \circ \beta( \alpha'_1 t'_1 \graft \dots \alpha'_m t'_m \graft \delta(\alpha' t') - \alpha' ( \alpha'_1 t'_1 \graft \dots \graft \alpha'_m t'_m \graft \delta(t')  )  ).
	\end{align*}
	For compositions $\mathcal{MPT} \circ \Psi$, we first have to show equation \eqref{eq::Freeness1} for the special case $\phi \in \mathcal{MPT}$. Let $\phi = \delta \alpha t$:
	\begin{align*}
		\beta(\delta \alpha t)(\zeta(\alpha' t'))=&\delta (\zeta(\alpha t)) (\zeta(\alpha' t') )\\
		=& \zeta(\alpha' t') \graft \zeta(\alpha t )\\
		=& \zeta(\alpha' t' \graft \alpha t  ) \\
		=& \zeta( \beta( \delta \alpha t)(\alpha' t')  ).
	\end{align*}
	Now suppose that identity \eqref{eq::Freeness1} holds for $\phi=\alpha_2 t_2 \graft \dots \graft \alpha_n t_n \graft \delta \alpha t$, we show that it holds for $\phi_2=\alpha_1 t_1 \graft \phi$:
	\begin{align*}
		\beta( \alpha t \graft \phi )(\zeta(\alpha' t'))=& \bigl( \zeta(\alpha_1 t_1) \graft \beta(\phi)\bigr)( \zeta(\alpha' t') ) \\
		=& \zeta(\alpha_1 t_1) \graft \beta(\phi)(\zeta(\alpha' t')) - \beta(\phi)(\zeta(\alpha_1 t_1 \graft \alpha' t'))\\
		=&\zeta(\alpha_1 t_1 \graft \phi(\alpha' t') - \phi(\alpha_1 t_1 \graft \alpha' t' ) ) \\
		=& \zeta\bigl( (\alpha_1 t_1 \graft \phi)(\alpha' t') \bigr).
	\end{align*}
	Then since $\beta$ respects compositions in $\mathcal{MPT}$, we get identity \eqref{eq::Freeness1} on $\mathcal{MPT}$. Using this, we can show the morphism property for $\beta$ on compositions $\mathcal{MPT} \circ \Psi$:
	\begin{align*}
		&\beta\bigl( (v_1,n_1,t_1) \circ (t' \to ((\alpha_1t_1 \ast \dots \ast \alpha_n t_n)t' \graft \alpha )t)  \bigr) \\
		=&\beta( t' \to ((\alpha_1t_1 \ast \dots \ast \alpha_n t_n)t' \graft \alpha )(t \graft_{v_1,n_1} t_1)   ) \\
		=& Z \to \bigl( ( \zeta(\alpha_1t_1)\ast \dots \ast \zeta(\alpha_nt_n) )Z \graft \gamma(\alpha)  \bigr) \zeta(t \graft_{v_1,n_1} t_1) \\
		=& Z \to \bigl( ( \zeta(\alpha_1t_1)\ast \dots \ast \zeta(\alpha_nt_n) )Z \graft \gamma(\alpha)  \bigr) \beta(v_1,n_1,t_1)(\zeta(t)) \\
		=&\beta(v_1,n_1,t_1) \circ \beta(t' \to ((\alpha_1t_1 \ast \dots \ast \alpha_n t_n)t' \graft \alpha )t).
	\end{align*}
	Hence $\beta$ is a composition morphism on all of $E\ell_{S(\mathcal{PA}_{\mathcal{C}})}(\mathcal{APT}_{\mathcal{C}})$. As a corollary, we now have identities \eqref{eq::Freeness1},\eqref{eq::Freeness2} everywhere. For identity \eqref{eq::Freeness3}, we can apply exactly the same computations as were done in \cite[Theorem 3.5, part c iii]{FloystadManchonMuntheKaas2020}. The uniqueness of $\zeta|_{Lie(\mathcal{PT}_{\mathcal{C}})}$ follows from $Lie(\mathcal{PT}_{\mathcal{C}})$ being the free post-Lie algebra. This then uniquely determines $\beta|_{\mathcal{MPT}_{\mathcal{C}}}$ by construction, which uniquely defines $\gamma$ by Lemma \ref{Lemma::PAisMPTQuotient}. The uniqueness of $\gamma$ then ensures uniqueness of $\zeta$ on the whole space $\mathcal{APT}_{\mathcal{C}}$.
\end{proof}

\bibliographystyle{acm}
\bibliography{PostLieRinehartReferences}
\end{document}